%% file: Infinite_FBMHs_final.tex
\documentclass[11pt]{amsart}
\usepackage{amssymb}
\usepackage{amsfonts}
\usepackage{amsmath}
\usepackage{graphicx}
\usepackage{color}
\usepackage{amsmath}
\usepackage{mathrsfs}
\usepackage{stmaryrd}
\usepackage{epsfig,color}
\usepackage{blindtext}
\usepackage{enumerate}
\usepackage{hyperref}
\usepackage{url}
\usepackage{bbm}
\usepackage{filecontents}
\usepackage{nicefrac,mathtools}
\DeclareGraphicsExtensions{.pdf,.jpeg,.png}
\usepackage{epstopdf}
\usepackage{cancel} 
\usepackage[normalem]{ulem} 
\usepackage{verbatim} 

\usepackage{color}
\usepackage[msc-links, lite]{amsrefs}
\usepackage{geometry}
\newtheorem{theorem}{Theorem}[section]
\newtheorem{conjecture}[theorem]{Conjecture}
\newtheorem{proposition}[theorem]{Proposition}
\newtheorem{lemma}[theorem]{Lemma}

\newtheorem{question}[theorem]{Question}

\newtheorem{claim}[]{Claim}

\theoremstyle{definition}
\newtheorem{definition}[theorem]{Definition}
\theoremstyle{remark}
\newtheorem{remark}[theorem]{Remark}

\numberwithin{equation}{section}

\newcommand{\dv}{\mathrm{div}}

\newcommand{\mf}{\mathbf}
\newcommand{\mb}{\mathbb}
\newcommand{\mc}{\mathcal}
\newcommand{\ms}{\mathscr}
\newcommand{\mk}{\mathfrak}

\newcommand{\wti}{\widetilde}

\newcommand{\Area}{\mathrm{Area}}

\newcommand{\Si}{\Sigma}

\newcommand{\ppt}{\frac{\partial}{\partial t}}
\newcommand{\pps}{\frac{\partial}{\partial s}}
\newcommand{\Xoe}{{X_\perp^\epsilon}}
\newcommand{\Xpe}{{X_\parallel^\epsilon}}

\newcommand{\dve}{\dv^\epsilon}
\newcommand{\vte}{\vartheta_\epsilon}
\newcommand{\he}{{h_\epsilon}}

\newcommand{\dist}{\operatorname{dist}}

\newcommand{\n}{\mathbf n}

\DeclareMathOperator{\Ric}{Ric}

\DeclareMathOperator{\spt}{spt}

\title[Infinitely many free boundary minimal hypersurfaces]{Existence of infinitely many free boundary minimal hypersurfaces}

\date{\today}

\author{Zhichao Wang}
\address{Max-Planck Institute for Mathematics, Vivatsgasse 7, 
53111 Bonn, Germany}
\email{wangzhichaonk@gmail.com}

\begin{document}

\begin{abstract}
In this paper, we prove that in any compact Riemannian manifold with smooth boundary, of dimension at least 3 and at most 7, there exist infinitely many almost properly embedded free boundary minimal hypersurfaces. This settles the free boundary version of Yau's conjecture. The proof uses adaptions of A. Song's work and the early works by Marques-Neves in their resolution to Yau's conjecture, together with Li-Zhou's regularity theorem for free boundary min-max minimal hypersurfaces.
\end{abstract}

\maketitle

\section{Introduction}
\subsection{Motivation from closed Riemannian manifolds}
Finding out minimal submanifolds has always been an important theme in Riemannian geometry. In 1960s, Almgren \citelist{\cite{Alm62}\cite{Alm65}} initiated a variational theory to find minimal submanifolds in any compact Riemannian manifolds (with or without boundary). He proved that weak solutions, in the sense of stationary varifolds, always exist. About twenty years later, the interior regularity theory for codimension one hypersurfaces was developed  by Pitts \cite{Pi} and Schoen-Simon \cite{SS}. As a consequence, they
showed that in any closed manifold $(M^{n+1} ,g)$, there exists at least one embedded closed minimal hypersurface, which is smooth except possibly along a singular set of Hausdorff codimension at least 7. Then Yau conjectured the following:
\begin{conjecture}[S.-T. Yau \cite{Yau82}]\label{conj:yau}
Every closed three-dimensional Riemannian manifold $(M^3, g)$ contains
infinitely many (immersed) minimal surfaces.
\end{conjecture}

The first progress of this Yau's Conjecture \ref{conj:yau} was made by Marques-Neves in \cite{MN17}, where they proved the existence of infinitely many embedded minimal hypersurfaces for closed manifolds with positive Ricci curvature, or more generally, for closed manifolds satisfying the ``Embedded Frankel Property''. Using the Weyl Law for the volume spectrum \cite{LMN16}, Irie-Marques-Neves \cite{IMN17} proved Yau's conjecture for generic metrics. Recently, in a remarkable work \cite{Song18}, A. Song completely solved the Conjecture \ref{conj:yau} building on the methods developed by Marues-Neves \citelist{\cite{MN17}\cite{MN16}}. Such a method also helped Song give a much stronger theorem: every closed Riemannian manifold $(M^{n+1},g)$ of dimension $3\leq (n+1)\leq 7$ contains infinitely many embedded minimal hypersurfaces.

\vspace{0.5em}
\subsection{Questions and Main results in compact Riemannian manifolds with boundary}
In this paper, we consider compact manifolds with boundary $(M,\partial M,g)$, which is the program set out by Almgren in the hypersurface
case \citelist{\cite{Alm62}\cite{Alm65}}. Then each critical point of the area functional is so called a {\em free boundary minimal hypersurface}, which is a hypersurface with vanished mean curvature and meeting $\partial M$ orthogonally along its boundary. Based on previous works \citelist{\cite{Pi}\cite{SS}}, Li-Zhou \cite{LZ16} proved the regularity on the free boundary, which implies the existence of free boundary minimal hypersurfaces in general compact manifolds with boundary.

Based on this regularity result, it is natural to raise a question bringing free boundary version of Yau's conjecture:
\begin{question}\label{ques:infinite many fbmhs}
Does every compact Riemannian manifold with smooth boundary of dimension $3\leq (n+1)\leq 7$ contain infinitely many free boundary minimal hypersurfaces?
\end{question}

Inspired by \citelist{\cite{MN16}\cite{IMN17}}, the author together with Guang, Li and Zhou proved the denseness of free boundary minimal hypersurfaces in compact manifolds with smooth boundary for generic metrics in \cite{GLWZ19}. Moreover, the author also proved that those free boundaries are dense in the boundary of the manifold; see \cite{Wang19_2}. In this paper, we settle Question 
\ref{ques:infinite many fbmhs} by adapting the arguments in \cite{Song18}.
\begin{theorem}
In any compact Riemannian manifold with boundary $(M^{n+1},\partial M,g)$, of dimension $3\leq (n+1)\leq 7$, there exist infinitely many almost properly embedded free boundary minimal hypersurfaces.
\end{theorem}

In this paper, we also use the growth of min-max width, which was firstly studied by Gromov \cite{Gro03} and \cite{Guth09} and quantified by  Liokumovich-Marques-Neves in \cite{LMN16}. According to the regularity theory in \citelist{\cite{Pi}\cite{SS}\cite{LZ16}}, each width is associated with an almost properly embedded free boundary minimal hypersurfaces with multiplicities; see \cite{GLWZ19}*{Proposition 7.3}. If each multiplicity is one, then since the widths are a sequence of real numbers going to infinity, it would lead to a direct proof of Yau's conjecture in the generic case. This is conjectured by Marques-Neves \cite{MN16}, and has been completely proven by Zhou \cite{Zhou19} for closed manifolds; see also Chodosh-Mantoulidis \cite{CM20} for three-manifolds of the Allen-Cahn version. However, such a kind of question remains open for compact manifolds with boundary.

We also mention there are other approaching to Question \ref{ques:infinite many fbmhs} in some special compact Riemannian manifolds with boundary. In the three dimensional round ball $\mb B^3$, Fraser-Schoen \cite{FrSch16} obtained the free boundary minimal surface with genus 0 and arbitrary many boundary components. By desingularization of the critical catenoid and the equatorial disk, Kapouleas-Li \cite{KL17} constructed infinitely many new free boundary minimal surfaces which have large genus in $\mb B^3$. We refer to \cite{Li19} for more results in $\mb B^3$.

\vspace{0.5em}
\subsection{Difficulties}
Compared to closed manifolds, the new main challenge is that in compact Riemannian  manifolds with boundary, the free boundary minimal hypersurfaces may have non-empty {\em touching sets} (see Definition \ref{def:touching}). Such touching phenomena always bring the main difficulties in the study of related problems; see \citelist{\cite{LZ16}\cite{ZZ17}\cite{ZZ18}\cite{GWZ_2}\cite{GWZ18}\cite{GLWZ19}\cite{Wang19}}. Precisely, if cutting a manifold along an almost free boundary minimal hypersurface with non-empty touching set, the result would never be a manifold even in the topological sense. In this paper, we come up with several new concepts (see Section \ref{sec:pre for fbmh}) and develop the ``embedded Frankel property'' in several ways (see Subsection \ref{subsec:constr of area minimizer} and Theorem \ref{thm:infitely many fbmhs}) which may be helpful in the further studies.

\medskip
Another challenge is the regularity of free boundary minimal hypersurfaces produced by min-max theory in compact manifolds whose boundaries are not smooth. We mention that there is no such regularity even for minimizing problems, which would be quite crucial for the smoothness of {\em replacements} (see \cite{LZ16}*{Proposition 6.3}). Nevertheless, we get the full regularity in our situation (see Theorem \ref{thm:min-max for Nepsilon}) by noticing that Li-Zhou's \cite{LZ16} result holds true for all smooth boundary points.

\vspace{0.5em}
\subsection{Outline of the proof}
Let $(M^{n+1},\partial M,g)$ be a compact Riemannian manifold with non-empty boundary, of $3\leq (n+1)\leq 7$. Assume that $(M,\partial M,g)$ contains only finitely many almost properly embedded free boundary minimal hypersurfaces. Borrowing the idea from Song \cite{Song18}, we notice that there are two key points:
\begin{itemize}
\item cutting along stable free boundary minimal hypersurfaces to get a connected component $N$ so that the free boundary minimal hypersurfaces in $N\setminus T$ (here $T$ is the new boundary part from cutting process) satisfy the Frankel property;
\item producing almost properly embedded free boundary minimal hypersurfaces in $N\setminus T$  by using min-max theory for $\mc C(N)$, which is a non-compact manifold by gluing to $N$ the cylindrical manifold $T\times [0,+\infty)$ under the conformal metric.
\end{itemize}

For the first part, we have to cut along the improper hypersurfaces, which would never lead the new thing to be a manifold even in the topological sense. To overcome this, we choose an order of those hypersurfaces carefully so that every time there is a connected component which is a compact manifold with piecewise smooth boundary satisfying our condition. Precisely, we cut along stable, properly embedded free boundary minimal hypersurfaces first and take a connected component $(N_1,\partial N_1, T_1,g)$ ($T_1$ is the new boundary part from cutting process) so that there is no stable properly embedded one in $N_1\setminus T_1$. Then each almost properly embedded free boundary minimal hypersurfaces in $N_1\setminus T_1$ {\em generically separates} $N_1$ (see Subsection \ref{subsec:constr of area minimizer}). If $N_1$ doesn't satisfy the Frankel property, then we prove that there exists a free boundary minimal hypersurface $\Sigma$ so that one of the connected component of $T_1\setminus \Sigma$ is good enough for us; see Lemma \ref{lem:general existence of good fbmh}.

\medskip
For the second part, we approach $\mc C(N)$ by a sequence of compact manifold with piecewise smooth boundary $N_\epsilon$. The key observation is that Li-Zhou's regularity holds true for all smooth boundary points. Hence we can use the monotonicity formula \citelist{\cite{GLZ16}*{Theorem 3.4}\cite{Si}*{\S 17.6}} to show that for any $p$ fixed, any $\epsilon>0$ small enough, the width $\omega_p(N_\epsilon)$ is associated with a properly embedded free boundary minimal hypersurface whose boundary lies on $N_\epsilon\cap \partial M$; see Theorem \ref{thm:min-max for Nepsilon} for details.

\medskip
This paper is organized as follows. In Section \ref{sec:pre for fbmh}, we give basic definitions and prove a generalized Frankel property for free boundary minimal hypersurfaces in the end. Then in Section \ref{sec:confined min-max theory}, we prove a min-max theory for a non-compact manifold with boundary. Finally, we prove the main theorem in Section \ref{sec:proof of main thm}. In Appendix \ref{sec:app:MP}, we state a strong maximum principle for stationary varifolds in compact manifolds with boundary and also sketch the proof. Appendix \ref{sec:app:Vinfty is stationary} contains the collection of the calculation in Theorem \ref{thm:min-max for cmpt with ends}.

\vspace{1em}
\subsection*{Acknowledgments:} The author would like to thank Prof. Xin Zhou for sharing his insights in minimal surfaces to me and many helpful discussion. The author would also like to thank Antoine Song for his explanation on \citelist{\cite{Song18}*{Lemma 10}\cite{BBN10}*{Proposition 5}}.

\section{Preliminary for free boundary minimal hypersurfaces}\label{sec:pre for fbmh}
In this section, we give the basic notations and some lemmas about constructing area minimizers in compact manifolds with boundary. 

Throughout this paper, $(M^{n+1},\partial M,g)$ is always a compact Riemannian manifold with smooth boundary and $3\leq (n+1)\leq 7$. Generally, $(M,\partial M,g)$ can be regarded as a domain of a closed Riemannian manifold $(\wti M,g)$. We also need to consider compact manifold with piecewise smooth boundary.
\begin{definition}[\cite{GWZ_2}*{Definition 2.2}]
\label{D:manifold with boundary and portion}
For a manifold with piecewise smooth boundary, $N$ is called a \emph{manifold with boundary $\partial N$ and portion $T$} if
\begin{itemize}
  \item $\partial N$ and $T$ are smooth, which may be disconnected;
  \item $\partial N\cup T$ is the (topological) boundary of $N$ and $\partial N\cap T=\partial T$.
\end{itemize}
We will denote it by $(N,\partial N,T)$.
\end{definition}

We remark that in the above definition, the interior of $\partial N$ and $T$ are disjoint.

\begin{definition}[\cite{LZ16}*{Definition 2.6}]\label{def:touching}
Let $(N^{n+1},\partial N,T,g)$ be a compact Riemannian manifold with boundary and portion. Let $\Sigma^n$ be a smooth $n$-dimensional manifold with boundary $\partial \Sigma$. We say that a smooth embedding $\phi: \Sigma\to N$ is an \emph{almost properly embedding} of $\Sigma$ into $N$ if $\phi(\Sigma)\subset N$ and $\phi(\partial \Sigma)\subset \partial N$. We say that $\Sigma$ is an {\em almost properly embedded hypersurface} in $N$.
\end{definition}

For an almost properly embedded hypersurface $(\Sigma, \partial\Sigma)$, we allow the interior of $\Sigma$ to touch $\partial N$. That is to say: $\mathrm{Int}(\Sigma)\cap \partial N$ may be non-empty.  We usually call $\mathrm{Int}(\Sigma)\cap \partial N$ the {\em touching set} of $\Sigma$.

\begin{definition}[\cite{LZ16}*{Section 2.3}]
Let $(\Sigma,\partial \Sigma)$ be an almost properly embedded hypersurface in $(N,\partial N,T,g)$. Then $\Sigma$ is called a {\em free boundary minimal hypersurface} if the mean curvature vanishes everywhere and $\Sigma$ meets $\partial N$ orthogonally along $\partial \Sigma$. 

We also use the term of {\em free boundary hypersurface} if $\Sigma$ only meets $\partial N$ orthogonally along $\partial \Sigma$.
\end{definition}

In this paper, we also need to deal with free boundary hypersurfaces which have touching sets from only one side.
\begin{definition}\label{def:half-properly embedded}
A two-sided embedded free boundary hypersurface $(\Sigma,\partial \Sigma)$ in $(N,\partial N,T,g)$ is {\em half-properly embedded} if it is almost properly embedded and has a unit normal vector field $\n$ so that $\n=\nu_{\partial M}$ along the touching set of $\Gamma$.
\end{definition}

\subsection{Neighborhoods foliated by free boundary hypersurfaces }
Given a metric on $N$, $(N,\partial N,T,g)$ can always be isometrically embedded into a compact Riemannian manifold with smooth boundary $(M,\partial M,g)$. Also, we embed $(M,\partial M,g)$ isometrically into a smooth Riemannian manifold $(\wti M,g)$ which has the same dimension with $M$ and $N$. Let $\Gamma$ be a two-sided, almost properly embedded, free boundary hypersurface in $(N,\partial N,T,g)$, Then $X\in\mathfrak X(\wti M)$ is called {\em an admissible vector field on $\wti M$ for $\Gamma$} if $X|_\Gamma$ is a normal vector field of $\Gamma$ and $X(p)\in T_p(\partial M)$ for $p$ in some neighborhood of $\partial \Gamma$ in $\partial M$. Note that such an admissible vector field is always associated with a family of diffeomorphisms of $\wti M$.
\begin{lemma}\label{lem:good neighborhood for non-degenerate}
Let $\Gamma$ be an almost properly embedded, two-sided non-degenerate free boundary minimal hypersurface in $(N,\partial N,T,g)$ and $\n$ a choice of unit normal vector on $\Gamma$. Let $\{\Phi(\cdot,t)\}_{-1\leq t\leq 1}$ be a family of diffeomorphisms of $\wti M$ associated to an admissible vector field on $\wti M$ for $\Gamma$ so that $\frac{\partial \Phi(x,t)}{\partial t}|_{t=0,x\in \Gamma}=\n(x)$. Then there exist a positive number $\delta_1$ and a smooth map $w:\Gamma\times(-\delta_1,\delta_1)\rightarrow \mb R$ with the following properties:
\begin{enumerate}
\item for each $x\in\Gamma$, we have $w(x, 0) = 0$ and $\phi:=\frac{\partial}{\partial t}w(x, t)|_{t=0}$ is a positive function which is the first eigenfunction of the second variation of area on $\Gamma$;
\item for each $t\in(-\delta_1 , \delta_1 )$, we have $\int_{\Gamma}(w(\cdot,t) -t\phi )\phi = 0$;
\item for each $t\in (-\delta_1 ,\delta_1)\setminus \{0\}$, $\{\Phi(x, w(x, t)): x \in\Gamma \}$ is an embedded hypersurface in $\wti M$ with free boundary on $\partial M$ and mean curvature either positive or negative.
\end{enumerate}
\end{lemma}

Lemma \ref{lem:good neighborhood for non-degenerate} follows from the implicit function theorem. With more effort, we have a similar result for degenerate stable free boundary minimal hypersurfaces.
\begin{lemma}\label{lemma:good neighborhood}
Let $\Gamma$ be an almost properly embedded, two-sided degenerate stable free boundary minimal hypersurface in $(M,\partial M, g)$ and $\n$ a choice of unit normal vector on $\Gamma$. Let $\{\Phi(\cdot,t)\}_{-1\leq t\leq 1}$ be a family of diffeomorphisms of $\wti M$ associated to an admissible vector field on $\wti M$ for $\Gamma$ so that $\frac{\partial \Phi(x,t)}{\partial t}|_{t=0,x\in \Gamma}=\n(x)$. Then there exist a positive number $\delta_1$ and a smooth map $w:\Gamma\times(-\delta_1,\delta_1)\rightarrow \mb R$ with the following properties:
\begin{enumerate}
\item\label{item:derivative} for each $x\in\Gamma$, we have $w(x, 0) = 0$ and $\phi:=\frac{\partial}{\partial t}w(x, t)|_{t=0}$ is a positive function in the kernel of the Jacobi operator of $\Gamma$;
\item \label{item:orthogonal to phi} for each $t\in(-\delta_1 , \delta_1 )$, we have $\int_{\Gamma}(w(\cdot,t) -t\phi )\phi = 0$;
\item\label{item:each slice cmc} for each $t\in (-\delta_1 ,\delta_1)$, $\{\Phi(x, w(x, t)): x \in\Gamma \}$ is an embedded hypersurface in $\wti M$ with free boundary on $\partial M$ and mean curvature either positive or negative or identically zero.
\end{enumerate}
\end{lemma}
\begin{proof}
The proof here is similar to \citelist{\cite{BBN10}*{Proposition 5}\cite{Song18}*{Lemma 10}}.

Denote the space
\[Y:=\{f\in C^\infty(\Gamma):\int_{\Gamma}f\phi=0\}.\]
Define a map $\Psi:Y\times \mb R\rightarrow Y\times C^{\infty}(\partial \Gamma)$ by 
\[\Psi(f,t)=\Big(\phi^{-1}[H(\Phi(x,f+t\phi))-\frac{1}{\Area(\Gamma)}\int_{\Gamma}H(\Phi(x,f+t\phi))],\langle \n(\Phi(x,f+t\phi)),\nu_{\partial M}\rangle|_{\partial \Gamma}\Big).\]
Then the first derivative (see \cite{GWZ_2}*{Lemma 2.5}) is 
\[D\Psi_{(0,0)}(f,0)=\Big(\phi^{-1}(Lf-\frac{1}{\Area(\Gamma)}\int_{\Gamma}Lf), fh^{\partial M}(\n,\n)-\langle\nabla f,\nu_{\partial M}\rangle|_{\partial \Gamma}\Big)\]
Here $L=\Delta+\Ric (\n,\n)+|A|^2$ is the Jacobi operator. Hence $D_1\Psi_{(0,0)}f=0$ is equivalent to $Lf=c$ and $\frac{\partial f}{\partial \eta}=h^{\partial M}(\n,\n)f$ (where $\eta$ is the co-normal of $\Gamma$), which implies that $f=0$. Then by the implicit function theorem, for each $t\in (-\delta_1,\delta_1)$, there exists a function $u(\cdot, t)\in Y$ so that $\Psi(u,t)=(0,0)$. Now define $w(x,t)=u(x,t)+t\phi$. Clearly, $w$ satisfies \eqref{item:orthogonal to phi} and \eqref{item:each slice cmc}.

It remains to verify \eqref{item:derivative}. Indeed, according to the implicit function theorem, we also have
\[ D_1\Psi (\frac{\partial u}{\partial t})\Big|_{(0,0)}+D_2\Psi(\ppt)\Big|_{(0,0)}=0. \]
By the direct computation, $D_2\Psi(\ppt)\big|_{(0,0)}=0$. Recall that $D_1\Psi\big|_{(0,0)}$ is nondegenerate. Hence $\frac{\partial u}{\partial t}\Big|_{t=0}=0$, which implies the desired result.
\end{proof}

Let $S$ be a two-sided free boundary minimal hypersurface in an $(n+1)$-dimensional compact manifold $(\hat M,\partial\hat M)$ (possibly with portion). Let $\wti M$ be a closed Riemannian manifold so that $\hat M$ is a compact domain of $\wti M$. Let $\mu > 0$; consider a neighborhood $\mc N$ of $S$ in $\wti M$ and a diffeomorphism
\[\wti F: S\times (-\mu,\mu)\rightarrow \mc N\]
such that $\wti F(x,0) = x$ for $x\in S$. We define the following (cf. \cite{Song18}*{Section 3}):
\begin{itemize}
\item $S$ has a {\em contracting neighborhood} if there are such $\mu,\mc N$ and $\wti F$ such that for all $t\in [-\mu,\mu]\setminus \{0\}$, $\wti F(S\times\{t\})$ has free boundary and mean curvature vector pointing towards $S$;
\item $S$ has an {\em expanding neighborhood} if $S$ is unstable or there are such $\mu,\mc N$ and $\wti F$ such that for all $t\in [-\mu,\mu]\setminus \{0\}$, $\wti F(S\times\{t\})$ has free boundary and mean curvature vector pointing away from $S$;
\item $S$ has a {\em mixed neighborhood} if there are such $\mu,\mc N$ and $\wti F$ such that for all $t\in [-\mu,0)$ (resp. $t\in(0,\mu]$), $\wti F(S\times\{t\})$ has free boundary and mean curvature vector pointing away from (resp. pointing towards) $S$;
\item $S$ has a {\em contracting neighborhood in one side} if there are such $\mu$, $\mc N$ and $\wti F$ such that for all $t\in (0, \mu]$, $\wti F(S\times\{t\})$ has free boundary and mean curvature vector pointing towards $S$; such a neighborhood in one side is said to be {\em proper} if $\wti F(S\times\{t\})\subset \hat M$ for $t\in(0,\mu)$;
\item $S$ has an {\em expanding neighborhood in one side} if $S$ is unstable or there are such $\mu$, $\mc N$ and $\wti F$ such that for all $t\in (0, \mu]$, $\wti F(S\times\{t\})$ has free boundary and mean curvature vector pointing away from $S$; such a neighborhood in one side is said to be {\em proper} if $\wti F(S\times\{t\})\subset \hat M$ for $t\in (0,\mu)$.
\end{itemize}

Let $S$ be a one-sided free boundary minimal hypersurface in $(\hat M,\partial \hat M,g)$. Denote by $\wti S$ the double cover of $S$. Consider the double cover $(M',\partial M',g')$ of $(\hat M,\partial \hat M,g)$ so that $\wti S$ is a two-sided free boundary minimal hypersurface in it. Then we say that $S$ has {\em a contracting (resp. an expanding) neighborhood} if $\wti S$ has a contracting (resp. an expanding) neighborhood.

\begin{remark}\label{rmk:not contracting and proper in one side}
Let $S$ be a two-sided free boundary minimal hypersurface and $\wti F$ be the diffeomorphism as above. $S$ is called to have {\em no}  proper contracting neighborhood in one side provided that each neighborhood in one side is not contracting or non-proper, i.e. there exist two sequences of real numbers $t_i^+\rightarrow 0^+$ and $t_i^-\rightarrow 0^-$ so that for each $t_i=t_i^+$ or $t_i^-$,
\begin{itemize}
	\item either $\wti F(S\times\{t_i\})$ has mean curvature vector pointing away from $S$;
	\item or $\wti F(S\times\{ t_i \})\setminus \hat M\neq \emptyset$.
\end{itemize}
\end{remark}

\subsection{Construction of area minimizers}\label{subsec:constr of area minimizer}
Let $(N,\partial N,T,g)$ be a connected compact manifold with  boundary and portion.  Let $(\Sigma,\partial\Sigma)$ be an almost properly embedded hypersurface in $(N,\partial N,T,g)$. Recall that {\em $\Sigma$ generically separates $N$} (see \cite{GWZ_2}*{Section 5}) if  there is a cut-off function $\phi$ defined on $\Sigma$ satisfying the following:
\begin{itemize}
\item $\phi$ is compactly supported in $\Sigma\setminus \partial \Sigma$ such that $\langle \phi\n,\nu_{\partial M}\rangle <0$ on the touching set, where $\n$ is the normal vector field of $\Sigma$; 
\item $\Sigma_{t\phi}:=\{\mathrm{exp}_x(t\phi\n):x\in\Sigma\}$ separates $N$ for all sufficiently small $t>0$.
\end{itemize} 

If $\Sigma$ generically separates $N$, then $N\setminus \Sigma$ can be divided into two part by the signed distance function to $\Sigma$. These two parts are called the {\em generic components}. 

\medskip
In this section, we consider the following conditions of $(N,\partial N,T,g)$:
\begin{enumerate}[A)]
\item\label{assump:contracting portion} the portion $T$ is a free boundary minimal hypersurface in $(N,\partial N,g)$ and has a contracting neighborhood in one side;
\item\label{assump:generic sparation} each two-sided free boundary minimal hypersurface generically separates $N$; 
\item\label{assump:regular neighborhood} any properly embedded, two-sided, free boundary minimal hypersurface in $N\setminus T$ has a neighborhood which is either contracting or expanding or mixed;
\item\label{assump:half regular neighborhood} any half-properly embedded, two-sided, free boundary minimal hypersurface in $N\setminus T$ has a proper neighborhood \underline{in one side} which is either contracting or expanding;
\item\label{assump:one sided expande} each properly embedded, one-sided, free boundary minimal hypersurface has an expanding neighborhood;
\item\label{assump:at most one minimal component} at most one connected component of $\partial N$ is a closed minimal hypersurface, and if it happens, it has an expanding neighborhood in one side in $N$.
\end{enumerate}

Let $\Gamma_1$ and $\Gamma_2$ be two disjoint, connected free boundary minimal hypersurface in $(N,\partial N,T,g)$ with $\Gamma_j\subset N\setminus T$ ($j=1,2$).
\begin{proposition}\label{prop:two-sided:non-degenerate}
Assume that $(N,\partial N,T,g)$ satisfies \eqref{assump:contracting portion}, \eqref{assump:generic sparation} and \eqref{assump:at most one minimal component}. Suppose that $\Gamma_j$ ($j=1,2$) is two-sided, non-degenerate and has no proper contracting neighborhood in one side (see Remark \ref{rmk:not contracting and proper in one side}). Then there exists a properly embedded free boundary minimal hypersurface in $N\setminus T$ which is an area minimizer.
\end{proposition}
\begin{proof}
We first consider that $\Gamma_j$ is not contained in $\partial N$. Since $\Gamma_j$ is non-degenerate, then $\Gamma_j$ a contracting or expanding neighborhood (see Lemma \ref{lem:good neighborhood for non-degenerate}), i.e. there exist $\mu>0$, a neighborhood $\mc N_j$ of $\Gamma_j$ in $\wti M$, and a diffeomorphism 
\[\wti F^j:\Gamma_j\times (-\mu,\mu)\rightarrow\mc N_j\]
such that $\wti F^j(x,0)=x$ for $x\in\Gamma_j$ and for each $t\in (-\mu,\mu)\setminus\{0\}$, $\wti F^j(\Gamma_j\times\{t\})$ has free boundary and mean curvature vector pointing towards or away from $\Gamma_j$. By assumption \eqref{assump:generic sparation}, $\Gamma_1$ and $\Gamma_2$ generically separates $N$. Hence $N\setminus (\Gamma_1\cup\Gamma_2)$ has three generic components.  Let $N'$ be the closure of the generic component of $N\setminus(\Gamma_1\cup\Gamma_2)$ that contains $\Gamma_1$ and $\Gamma_2$. Without loss of generality, we assume that for $t>0$, $\wti F^j(\Gamma_j\times \{t\})$ intersects $N'\setminus (\Gamma_1\cup\Gamma_2)$.

 Now take $\epsilon\in(0,\mu)$ so that $\wti F^j(\Gamma_j\times \{\pm\epsilon\})$ meets $\partial N$ transversally for $j=1,2$.

\medskip
{\noindent\em Case 1: Both $\Gamma_1$ and $\Gamma_2$ have expanding neighborhoods.}

\medskip
In this case, we consider
\begin{gather*}
N_1:=N'\setminus\cup_{j=1}^2\wti F^j(\Gamma_j\times [0,\epsilon)),\ \  \
\partial N_1:=\partial N\cap N_1,\\
T_1:=\big[\cup_{j=1}^2\wti F^j(\Gamma_j\times \{\epsilon\})\cup T\big]\cap N'.
\end{gather*}
Clearly, $(N_1,\partial N_1,T_1,g)$ is a compact manifold with boundary and portion. Moreover, $\wti F^1(\Gamma_1\times\{\epsilon\})$ represents a non-zero relative homology class in $(N_1,\partial N_1)$. By minimizing the area of this class, we obtain a stable free boundary minimal hypersurface and a connected component $S$ is properly embedded in $N\setminus T$, which is the desired hypersurface since it is obtained by a minimizing procedure.

\medskip
{\noindent\em Case 2: Both $\Gamma_1$ and $\Gamma_2$ have contracting neighborhoods.}

\medskip
In this case, we consider
\begin{gather*}
N_2:=\cup_{j=1}^2\wti F^j(\Gamma_j\times [-\epsilon,0))\cup N',\\
\partial N_2:=(\partial N\cap N')\cup \cup_{j=1}^2\wti F^j(\partial\Gamma_j\times [-\epsilon,0)),\\
T_2:=\cup_{j=1}^2\wti F^j(\Gamma_j\times \{-\epsilon\})\cup (T\cap N').
\end{gather*}
Clearly, $(N_2,\partial N_2,T_2,g)$ is a compact manifold with boundary and portion (see Figure \ref{fig:both contract}). We can minimize the area of the relative homology class represented by $\wti F^1(\Gamma_1\times \{-\epsilon\})$ to get a free boundary minimal hypersurface. Particularly, one connected component is stable and properly embedded in $N\setminus T$ and is an area minimizer.

\begin{figure}[h]
\begin{center}
\def\svgwidth{0.78\columnwidth}
  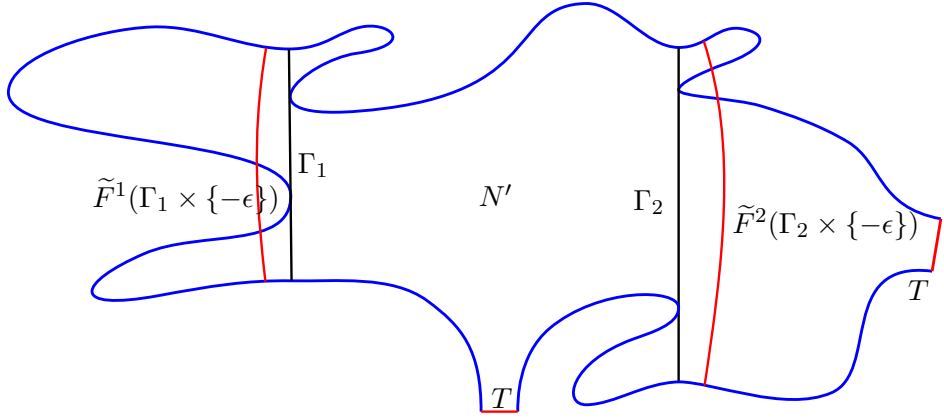
  \caption{Barriers from contracting neighborhoods.}
  \label{fig:both contract}
\end{center}
\end{figure}

\medskip
{\noindent\em Case 3: $\Gamma_1$ has a contracting neighborhood and $\Gamma_2$ has an expanding neighborhood.}

\medskip
In this case, we consider
\begin{gather*}
N_3:=N'\cup\wti F^1(\Gamma_1\times [-\epsilon,0))\setminus \wti F^2(\Gamma_2\times [0,\epsilon)),\\
\partial N_3:=(\partial N\cap N')\cup \wti F^1(\partial\Gamma_j\times [-\epsilon,0))\setminus \wti F^2(\Gamma_2\times [0,\epsilon)),\\
T_3:=\wti F^1(\Gamma_1\times \{-\epsilon\})\cup\big[(\wti F^2(\Gamma_2\times \{\epsilon\})\cup T)\cap N'\big].
\end{gather*}
By the same argument in the first two cases, we then obtain the desired hypersurface.

\medskip
To complete the proof, it suffices to consider $\Gamma_1\subset \partial N$. Then by assumption \eqref{assump:at most one minimal component}, $\Gamma_1$ has an expanding neighborhood in one side. Then it is just a subcase of Case 1 or Case 3. In either case, we can find a properly embedded, stable free boundary minimal hypersurface having a contracting neighborhood.
\end{proof}

\begin{remark}
In both Case 1 and 3, we used the expanding neighborhood in one side to be a barrier for minimizing problems even if such a neighborhood is not proper. The key observation here is that the interior of $\wti F^j(\Gamma_j\times\{\epsilon\})$ intersects $N'$ with angles less than $\pi/2$. We refer to \cite{GWZ_2}*{Lemma 4.13} for details.
\end{remark}

We now give a stronger proposition by a perturbation argument.
\begin{proposition}\label{prop:both two-sided}
Suppose that $(N,\partial N,T,g)$ satisfies (\ref{assump:contracting portion}--\ref{assump:at most one minimal component}).  Suppose that $\Gamma_j$ ($j=1,2$) is two-sided and has no proper contracting neighborhood in one side (see Remark \ref{rmk:not contracting and proper in one side}). Then there exists a two-sided, properly embedded, stable, free boundary minimal hypersurface having a contracting neighborhood.
\end{proposition}
\begin{proof}[Proof of Proposition \ref{prop:both two-sided}]
Firstly, we consider that $\Gamma_j$ is not part of $\partial N$ for $j=1,2$. Denote by 
$N'$ the closure of the generic component of $N\setminus (\Gamma_1\cup\Gamma_2)$ that contains $\Gamma_1$ and $\Gamma_2$. Let $\mc N_j$ be a neighborhood of $\Gamma_j$ in $\wti M$ and  
\[\wti F^j:\Gamma_j\times (-\mu,\mu)\rightarrow \mc N_j\]
be the map constructed by Lemma \ref{lemma:good neighborhood} for $\Gamma_j$ and $\mu>0$. Without loss of generality, we assume that for $t<0$,
\[\wti F^j(\Gamma_j\times \{t\})\cap N'=\emptyset.\]
Since $\Gamma_j$ does not have a proper and contracting neighborhood in one side, then by \eqref{assump:regular neighborhood} and \eqref{assump:half regular neighborhood}, each neighborhood in one side is expanding if it is proper. Hence in both cases, we can always take $\epsilon>0$ so that for $j=1,2$,  
\[\Area(\wti F^j(\Gamma_j\times\{\epsilon\})\cap N)<\Area (\Gamma_j).\]
Denote by 
\[\mc A_1:=\min_{j\in\{1,2\}}\Area(\wti F^j(\Gamma_j\times \{\epsilon\})).\]
Then we can take $r_k\rightarrow 0$, $q_j\in\Gamma_j\setminus\partial N$ so that
\[B_{r_k}(q_j)\cap \partial N=\emptyset \  \ \text{ and } \ \ B_{r_k}(q_j)\cap \wti F^j(\Gamma_j\times \{\epsilon\})=\emptyset.\] 
By \cite{IMN17}*{Proposition 2.3} (see also \cite{GWZ_2}*{Remark 5.5}), there exists a sequence of perturbed metrics $g_k\rightarrow g$ on $\wti M$ so that 
\begin{itemize}
\item $g_k(x)=g(x)$ for all $x\in\wti M\setminus (B_{r_k}(q_1)\cup B_{r_k}(q_2))$;
\item both $\Gamma_1$ and $\Gamma_2$ are non-degenerate free boundary minimal hypersurfaces in $(N,\partial N,T,g_k)$.
\end{itemize}
Clearly, $(N,\partial N,T,g_k)$ satisfies \eqref{assump:contracting portion}, \eqref{assump:generic sparation} and \eqref{assump:at most one minimal component}. Applying Proposition \ref{prop:two-sided:non-degenerate}, there exists a properly embedded free boundary minimal hypersurface $S_k$ which is an area minimizer. Moreover, by the argument in Proposition \ref{prop:two-sided:non-degenerate},
\[\Area_{g_k}(S_k)<\mc A_1.\]
Letting $k\rightarrow \infty$, by the compactness for stable free boundary minimal hypersurfaces \cite{GLZ16}, $S_k$ smoothly converges to a stable free boundary minimal hypersurface $S\subset N'$ in $(N,\partial N,T,g)$ with $\Area(S)\leq\mc A_1$. Such an area upper bound gives that $S$ is not $\Gamma_1$ or $\Gamma_2$. Then by the maximum principle, $S\cap B_{r_k}(q_j)=\emptyset$ for $j=1,2$. From the smooth convergence and the fact of $r_k\rightarrow 0$, $S_k\cap B_{r_k}(q_j)=\emptyset$ for large $k$. Hence for large $k$, $S_k$ is an area minimizer in $(N,\partial N,T,g)$. By the assumption \eqref{assump:one sided expande}, $S_k$ must be two-sided. Therefore, it is the desired free boundary minimal hypersurface in $(N,\partial N,T,g)$.

It remains to consider $\Gamma_1\subset \partial N$. By assumption \eqref{assump:at most one minimal component}, $\Gamma_1$ has an expanding neighborhood in one side in $N$. Then we just need to perturb the metric slightly near $\Gamma_2$. By a similar argument in Proposition \ref{prop:two-sided:non-degenerate}, we can also obtain a stable, properly embedded, free boundary minimal hypersurface with respect to perturbed metrics. Then the process above also gives a desired hypersurface.
\end{proof}

Now we are ready to state the main result in this section, which is a generalized Frankel property and will be used in the proof of Theorem \ref{thm:infitely many fbmhs}.
\begin{lemma}\label{lem:general existence of good fbmh}
Suppose that $(N,\partial N,T,g)$ satisfies (\ref{assump:contracting portion}--\ref{assump:at most one minimal component})and contains two disjoint connected free boundary minimal hypersurfaces in $N\setminus T$. Then $N\setminus T$ contains a two-sided, free boundary minimal hypersurface with a proper and contracting neighborhood \underline{in one side}.
\end{lemma}
\begin{proof}
Let $\Gamma_1$ and $\Gamma_2$ be two disjoint, free boundary minimal hypersurfaces in $(N,\partial N,T,g)$.

\medskip
{\noindent\em Case 1: $\Gamma_1$ and $\Gamma_2$ are both two-sided.}

\medskip
Without loss of generality, we assume that $\Gamma_1$ and $\Gamma_2$ have no proper contracting neighborhood in one side. Then this lemma follows from Proposition \ref{prop:both two-sided}.

\medskip
{\noindent\em Case 2: $\Gamma_1$ is two-sided and $\Gamma_2$ is one-sided.}

\medskip
Without loss of generality, we assume that $\Gamma_1$ has no proper contracting neighborhood in one side and $\Gamma_2$ is not properly embedded or has an expanding neighborhood. Now consider the double cover $(N_1,\partial N_1, T_1,g)$ of $(N,\partial N,T,g)$ so that the double cover $\wti \Gamma_2$ of $\Gamma_2$ is a two-sided free boundary minimal hypersurface in $(N_1,\partial N_1,T_1,g)$. Then applying Proposition \ref{prop:both two-sided} again, we obtain a two-sided, properly embedded, free boundary minimal hypersurface $S$ having a contracting neighborhood so that $S\subset N_1\setminus (\Gamma_1\cup\wti \Gamma_2)$. Clearly, $S$ is the desired hypersurface in $N\setminus T$.

\medskip
{\noindent\em Case 3: Both $\Gamma_1$ and $\Gamma_2$ are one-sided.}

\medskip
Consider the double cover $(N_2,\partial N_2,T_2,g)$ of $(N,\partial N,T,g)$ so that the double cover $\wti\Gamma_1$ of $\Gamma_1$ is a two-sided free boundary minimal hypersurface in $(N_2,\partial N_2,T_2,g)$. Then the desired result follows from Case 2.
\end{proof}

We finish this section by giving an area lower bound for the free boundary minimal hypersurfaces. This can also be seen as an application of the construction of area minimizer in Proposition \ref{prop:both two-sided}; cf. \cite{Song18}*{Lemma 12}.
\begin{lemma}\label{lem:area lower bound}
Suppose that $(N,\partial N,T,g)$ satisfies (\ref{assump:contracting portion}--\ref{assump:at most one minimal component}). Let $T_1,\cdots,T_q$ be the connected components of $T$. Assume that 
\begin{enumerate}[(i)]
\item\label{item:assume:proper expanding} every properly embedded free boundary minimal hypersurface in $N\setminus T$ has an expanding neighborhood;
\item\label{item:assume:half proper:expanding} every half-properly embedded free boundary minimal hypersurface in $N\setminus T$ has an expanding neighborhood in one side which is proper.
\end{enumerate}
Then for any free boundary minimal hypersurface $\Gamma$ in $N\setminus T$:
\begin{enumerate}
\item\label{item:two-sided:area lower bound} if $\Gamma$ is two-sided,
\[\Area(\Gamma)>\max\{\Area(T_1),\cdots,\Area(T_q)\};\]
\item\label{item:one-sided:area lower bound} if $\Gamma$ is one-sided,
\[2\Area(\Gamma)>\max\{\Area(T_1),\cdots,\Area(T_q)\}.\]
\end{enumerate}
\end{lemma}
\begin{proof}
We prove \eqref{item:two-sided:area lower bound} and then \eqref{item:one-sided:area lower bound} follows by considering the double cover. Without loss of generality, we assume that $T_1$ is the connected component of $T$ that has maximal area.

Assume on the contrary that $\Gamma$ is a two-sided free boundary minimal hypersurface in $N\setminus T$ so that 
\[\Area(\Gamma)\leq \max\{\Area(T_1),\cdots,\Area(T_q)\}.\]
Denote by $N'$ the closure of the generic component of $N\setminus \Gamma$ that contains $T_1$ and $\Gamma$. We divide the proof into two cases by considering whether $\Gamma$ has a proper neighborhood in $N'$ or not.

\medskip
If $\Gamma$ has a proper neighborhood in one side in $N'$, then such a neighborhood is expanding. Denote by 
\[\partial N'=\partial N\cap N'\  \ \text{ and }\  \ T'=(T\cap N')\cup \Gamma.\]
Then $(N',\partial N',T',g)$ is a compact manifold with boundary and portion. Clearly, $\Gamma$ represents a non-trivial relative homology class in $(N',\partial N')$. Using the argument in Proposition \ref{prop:both two-sided}, we obtain a two-sided, properly embedded, free boundary boundary minimal hypersurface $S$ having a contracting neighborhood. Note that $S$ does not contain $T'$ since 
\[\Area(S)<\Area(\Gamma)\leq\Area(T_1)\leq \Area(T').\]
Then $S$ has a connected component in $N'\setminus T'$, which contradicts the assumption \eqref{item:assume:proper expanding}.

\medskip
If $\Gamma$ has no proper neighborhood in one side in $N'$, then we can use a perturbation argument in Lemma \ref{lem:general existence of good fbmh} to construct an area minimizer $S'$ having $\Area(S')<\Area(\Gamma)$. Such an area bound also implies that $S$ contains a two-sided free boundary minimal hypersurface in $N'\setminus T$. This also contradicts \eqref{item:assume:proper expanding}.

\end{proof}

\subsection{No mass concentration at the corners}
In a compact manifold $(N^{n+1},\partial N,g)$ with smooth boundary (without portion), then the monotonicity formula in \cite{GLZ16} gives that a stationary $n$-varifold $V$ with free boundary 
can not support on an $(n-1)$-dimensional submanifold. In this subsection, we generalize this result directly in a compact manifold with boundary and non-empty portion.

\begin{lemma}\label{lem:no mass accumulate}
Let $(N,\partial N,T,g)$ be an $(n+1)$-dimensional compact manifold with boundary and portion. Let $V$ be an $n$-varifold such that the first variation vanishes along each vector field $X$ satisfying that $X$ is tangential when restricted on $\partial N$ and $T$. Denote by $S_V$ the support of $V$. Supposing that $T$ is a stable free boundary minimal hypersurface and $S_V\cap \partial T\neq \emptyset$, then for any neighborhood $U$ of $\partial T$ in $N$, $S_V$ intersects $U\setminus \partial T$.
\end{lemma}
\begin{proof}[Proof of Lemma \ref{lem:no mass accumulate}]
Recall that by Lemmas \ref{lem:good neighborhood for non-degenerate} and \ref{lemma:good neighborhood}, there exists a neighborhood $\mc N$ of $T$ and a diffeomorphism $F:T\times[0,\epsilon)\rightarrow \mc N$ so that $F(T\times\{t\})$ is an embedded free boundary hypersurface. 
	
Assume on the contrary that there exists a neighborhood $U$ of $\partial T$ so that $S_V\cap U\subset \partial T$. Without loss of generality, we assume that $U\subset \mc N$. Then $t\nabla t$ is tangential when restricted on $\partial N$ and $T$ in $U$. Hence the first variation of $V\llcorner U$ vanishes along $t\nabla t$, i.e.
	\begin{equation}\label{eq:for nabla s}
	0=\int \dv_S(t\nabla t)dV\llcorner U(x,S)=\int |p_S\nabla t|^2 dV\llcorner U(x,S). 
	\end{equation}
	Here $p_S(\cdot)$ is the projection to $S$.
	
Let $\mk d$ be the distance to $\partial T$ in $T$. We now extend it to be a function defined in a neighborhood of $T$ by setting 
	\[ \mk d(F(x,t)):=\mk d(x).  \]
	Then $\mk d\nabla \mk d $ is tangential when restricted on $\partial N$ and $T$ in $U$. Hence we also have
	\begin{equation}\label{eq:for nabla d}
	0=\int \dv_S(\mk d\nabla \mk d)dV\llcorner U(x,S)=\int |p_S\nabla \mk d|^2 dV\llcorner U(x,S).
	\end{equation}
	Note that $S$ is an $n$-dimensional hyperplane in $T_xN$ and $\nabla s\perp \nabla\mk d$ at $x\in  \partial T$. Hence there exists a unit vector $a$ so that 
	\[ a\in S\cap \{c_1\nabla s+c_2\nabla\mk d: c_1,c_2\in \mb R \}. \]
	Then 
	\[ |p_S\nabla t|^2 +|p_S\nabla\mk d|^2 \geq |g(\nabla s,a)|^2+|g(\nabla \mk d,a)|^2 \geq \min\{ |\nabla t|^2,|\nabla \mk d|^2 \}. \]
	Denote by $c_0=\min_{q\in  F(T\times [0,\epsilon])} \min\{|\nabla t|^2,1\}>0$. Note that $|\nabla \mk d|=1$ on $\partial T$ by definition. Hence we have that for $x\in\partial T$,
	\[  |p_S\nabla s|^2 +|p_S\nabla\mk d|^2 \geq c_0. \]
	However, this contradicts \eqref{eq:for nabla s} and \eqref{eq:for nabla d}. Hence Lemma \ref{lem:no mass accumulate} is proved.
	\end{proof}

We remark that in the Lemma \ref{lem:no mass accumulate}, the stability and minimality of $T$ are both redundant. It is only used to construct a neighborhood foliated by free boundary hypersurfaces. Such a result may hold true for any embedded hypersurface with free boundary by using the implicit function theorem.

\section{Confined min-max free boundary minimal hypersurfaces}\label{sec:confined min-max theory}
\subsection{Construction of non-compact manifold with cylindrical ends}\label{subsection:add cylinders}
In this part, we define the manifold with boundary and cylindrical ends. Then we will construct a sequence of compact manifold with boundary and portion converging to this non-compact manifold in some sense. The construction here is similar to \cite{Song18}*{Section 2.2} with necessary modifications.

Let $(N,\partial N,T,g)$ be a connected compact Riemannian manifold with boundary and portion endowed with a metric $g$. Suppose that $T$ is a free boundary minimal hypersurface. Then by Lemmas \ref{lem:good neighborhood for non-degenerate} and \ref{lemma:good neighborhood}, there a neighborhood of $T$ which is smoothly foliated with properly embedded leaves. In other words, there exist a neighborhood $\mc N$ of $T$ and a diffeomorphism
\[ F:T\times [\,0,\hat t\,]\rightarrow\mc N,\]
where $F(T\times \{0\})=T$ and for all $t\in(0,\hat t\,]$, $F(T\times\{t\})$ is a properly embedded hypersurface with free boundary. Moreover, there exists a positive function $\phi$ on $T$ so that 
\begin{equation}\label{eq:ppt at T}
F_*(\ppt)\Big|_T=\phi\n \text { and } g(F_*(\ppt),\n)>0 \text { on } F(T\times [0,\hat t]),
\end{equation}
where $\n$ is the unit inward normal vector field of $F(T\times\{t\})$ in $N$. 

In general, $F_*(\ppt)$ may not be a tangential vector field around $\partial T$ along $\partial N$. Fortunately, we can amend the diffeomorphism $F$ to overcome this (cf. \cite{MN16}*{Proof of Deformation Theorem C}). Note that $t$ can be used to define a function on $F(T\times [0,\hat t])$ by setting
\[ t (F(x,a)):=a. \]
Taking a vector field $X$ of $N$ so that it is an extension of $\nabla t/|\nabla t|^2$ and $X|_{\partial N}\in T(\partial N)$. Let $(\ms F_t)_{t\geq 0}$ be the one-parameter family of diffeomorphisms generated by $X$. We now claim that $\ms F_t(T)=F(T\times \{t\})$ for any $t\in[0,\hat t]$. Namely, for any $x\in T$, we have
\[ \frac{d}{ds}t(\ms F_s(x)) =\langle \nabla t,X\rangle=1. \]
Hence $t(\ms F_s(T))=s$ and the claim is proved. Note that $\ms F$ can also be regarded as a diffeomorphism 
\[  \ms F:T\times[\,0,\hat t\,]\rightarrow \mc N\]
by setting $\ms F(x,t)=\ms F_t(x)$. By the definition of $\ms F$ and $X$, we have that 
\begin{equation}\label{eq:amended ppt}
\ms F_*(\ppt)=\nabla t/|\nabla t|^2 
\end{equation}
for $t\in[0,\hat t]$. From \eqref{eq:ppt at T}, we also have
\begin{equation}\label{eq:new ppt at T}
\ms F_*(\ppt)\Big|_T=\phi\n \text { and } g(\ms F_*(\ppt),\n)>0 \text { on } \ms F(T\times [0,\hat t]).
\end{equation}
 We also use $\ppt$ to denote $\ms F_*(\ppt)$ for simplicity.

Clearly, there exists a positive smooth function $f$ on $\ms F(T\times [0,\hat t\,])$ so that the metric $g$ can be written as
\begin{equation}\label{eq:decompose metric}
g=g_t(q)\oplus (f(q)dt)^2, \ \ \ \forall q\in \ms F(T\times\{t\}). 
\end{equation}
Here $g_t=g\llcorner \ms F(T\times \{t\})$ is the restricted metric and it can be extend to define a 2-form over $TN$. Namely, a vector field $X$ can be decomposed to
\[  X=X_\perp+X_\parallel,\]
where $X_\perp$ is normal to $\ppt$ and $X_\parallel$ is a multiple of $\ppt$. Then for any two vector fields $X,Y$, we define
\[ g_t(X,Y):=g_t(X_\perp,Y_\perp). \]

We also remark that $f|_T=\phi$ by \eqref{eq:new ppt at T}.

 Now for any $\epsilon<\hat t$, define on $N$ the following metric $h_\epsilon$:
\[
h_\epsilon(q):=
\left\{
\begin{array}{ll}
g_t(q)\oplus(\vartheta_\epsilon(t)f(q)dt)^2 &\text{\ for \ } q\in \ms F(T\times[0, \epsilon])\\
g(q) &\text{\ for \ } q\in N\setminus \ms F(T\times [0, \epsilon]).
\end{array} 
\right.
\]
Here $\vartheta_\epsilon$ is chosen to be a smooth function on $[0,\epsilon]$ so that
\begin{itemize}
\item $1\leq \vartheta_\epsilon$ and $\frac{\partial}{\partial t}\vartheta_\epsilon\leq 0$;
\item $\vartheta_\epsilon\equiv1$ in a neighborhood of $\epsilon$;
\item $\lim_{\epsilon\rightarrow0}\int_{\epsilon/2}^\epsilon\vartheta_{\epsilon}=+\infty$;
\end{itemize}

Obviously, we have the following lemma.
\begin{lemma}[cf. \cite{Song18}*{Lemma 4}] \label{lem:2 form and mean curvatue in new metric}
Suppose that the leaf $\ms F(T\times \{t\})$ has free boundary on $\partial N$ and its non-zero mean curvature vector points towards $T$ for all $t\in(0,\epsilon)$ in $(N,\partial N,T,g)$. Then each slice $\ms F(T\times\{t\})$ is a free boundary hypersurface and satisfies the following with respect to the new metric $h_\epsilon$:
\begin{enumerate}
\item it has non-zero mean curvature vector pointing in the direction of $-\frac{\partial}{\partial t}$;
\item its mean curvature goes uniformly to zero as $\epsilon$ converges to 0;
\item its second fundamental form is bounded by a constant $C$ independent of $\epsilon$.
\end{enumerate} 
\end{lemma}

Let $\varphi:T\times \{0\}\rightarrow T$ be the canonical identifying map. Define the following non-compact manifold with cylindrical ends:
\[\mc C(N):=N\cup_{\varphi}(T\times [0,+\infty)).\]
We endow it with the metric $h$ such that $h=g$ on $N$ and 
\begin{equation}\label{eq:define of h}
h=g\llcorner T\oplus (f_0dt)^2
\end{equation}
on $T\times [0,+\infty)$. Here $g\llcorner T$ is the restriction of $g$ to the tangent bundle of $T$ and 
\[f_0(x,t)=f(\varphi(x,0))=\phi(\varphi(x));\]
see \eqref{eq:decompose metric} for the definition of $f$. We remark that under the metric $h$, each slice $T\times\{t\}$ is totally geodesic. We define the homeomorphism $\gamma:T\times (-\hat t,\hat t\,)\rightarrow \ms F(T\times[0,\hat t\,))\cup_\varphi T\times(-\hat t,0]$ by 
\begin{equation}\label{eq:define gamma}
\gamma(x,t)=\ms F(x,t) \text{ for } t\in[0,\hat t); \ \ \ \gamma(x,t)= (x,t) \text{ for } t\in (-\hat t,0) .
\end{equation}

The following lemma gives that $(\mc C(N),h)$ is a $C^1$ manifold.
\begin{lemma}\label{lem:CN is C1}
With the differential structure associated with $h$ on $N$ and $T\times (0,+\infty)$, $\mc C(N)$ is a $C^1$ manifold with boundary. Moreover, the metric is Lipschitz continuous. 
\end{lemma}
\begin{proof}
To complete the proof, it suffices to prove that $\gamma $ is a $C^1$ map. Note that $\gamma$ is smooth for $x$ everywhere. Since $\ms F$ and $\varphi$ are diffeomorphisms, $\gamma(x,\cdot)$ is smooth for $t\neq 0$. Now we consider its behavior at $t=0$. On the one hand, by \eqref{eq:new ppt at T}, 
\[ \lim_{t\rightarrow 0^+}\gamma_*(\ppt)=\lim_{t\rightarrow 0^+}\ms F_*(\ppt)=\phi \n. \]
On the other hand, for $t<0$, $\varphi_*(\ppt)$ is parallel to the tangent space the level set $\varphi(T\times \{t\})$ and it has the norm
\[  \Big|\varphi_*(\ppt)\Big|_h =f_0,\]
which implies that
\[ \lim_{t\rightarrow 0^-}\gamma_*(\ppt)=\lim_{t\rightarrow 0^-}\varphi_*(\ppt)=f_0\n=\phi\n.\]
Thus we conclude that $\mc C(N)$ is $C^1$. 

Note that the metric is smooth on $N$ and $T\times [0,+\infty)$. Hence it is Lipschitz on $\mc C(N)$. This completes the proof of Lemma \ref{lem:CN is C1}.
\end{proof}

The following Lemma shows that $(N,\partial N,T,h_\epsilon)$ converges to the non-compact manifold with cylindrical ends $\mc C(N)$ and the convergence is smooth away from $\ms F(T\times \{0\})$.
\begin{lemma}[cf. \cite{Song18}*{Lemma 5}]\label{lem:convergence to CN}
Let $q$ be a point of $N\setminus T$. Then $(N,h_\epsilon,q)$ converges geometrically to $(\mc C(N),h,q)$ in the $C^0$ topology as $\epsilon\to 0$. Moreover, the geometric convergence is smooth outside of $T\subset \mc C(N)$ in the following sense:
\begin{enumerate}
\item Let $q\in N\setminus \ms F(T\times [0,\hat t])$. Then as $\epsilon\rightarrow 0$,
\[(N\setminus\ms F(T\times[0,\epsilon]),h_\epsilon,q)\]
converges geometrically to $(N\setminus T,g,q)$ in the $C^\infty$ topology;
\item Fix any connected component $T_1$ of $T$. Let $q_\epsilon\in\ms F(T_1\times [0,\epsilon])$ be a point at fixed distance $\hat d>0$ from $\ms F(T_1\times \{\epsilon\})$ for the metric $h_\epsilon$, $\hat d$ being independent of $\epsilon$. Then 
\[(\ms F(T_1\times [0,\epsilon)),h_\epsilon,q_\epsilon)\]
subsequently converges geometrically to $(T_1\times(0,+\infty),h,q_\infty)$ in the $C^\infty$ topology, where $h$ is defined as \eqref{eq:define of h}, and $q_\infty$ is a point of $T_1\times (0,+\infty)$ at distance $\hat d$ from $T_1\times \{0\}$.
\end{enumerate}
\end{lemma}
\begin{proof}
The first item follows from that $\he=g$ on $N\setminus\ms F(T\times [0,\epsilon])$. We now prove the second one. Define the coordinate $s$ by
\[s(\ms F(x,t)):=-\int_t^\epsilon\vte(u)du.\]
Then $(\vartheta^2_\epsilon dt)^2=ds^2$ and $|\nabla^\epsilon s|_\he=|\nabla t|_g$. Hence 
\[ \he(q)=g_s(q)\oplus (f(q)ds)^2.\]	
Note that $g_t(\ms F(x,t))\rightarrow g_0(\ms F(x,0))=g\llcorner T$ and $f(\ms F(x,t))\rightarrow f(\varphi(x,0))$ as $\epsilon\rightarrow0$. Thus we conclude that 
\[ \he(\ms F(x,s))\rightarrow g\llcorner T\oplus (f_0ds)^2=h.\]
	\end{proof}

In the following lemma, we describe more about the convergence in a neighborhood of $\ms F(T\times \{0\})$ by finding explicit local charts.
\begin{lemma}[cf. \cite{Song18}*{Lemma 6}]
There exists $\eta>0$ such that for each $\epsilon\in (0,\hat t/2)$ small, there is an embedding $\sigma_\epsilon: T\times [-\hat t/2,\hat t/2]\rightarrow N$ satisfying the following properties:
\begin{enumerate}
\item $\sigma_\epsilon(T\times \{0\})=\ms F(T\times \{\epsilon\})$ and 
\[ \sigma_\epsilon (T\times [-\hat t/2 ,\hat t/2])=\{q\in N: |s(q)|\leq \hat t/2\} ;\]
\item $\|\sigma_\epsilon^*\he\|_{C^1(T\times [-\hat t/2,\hat t/2])}<\eta$,  where $\|\cdot\|_{C^1(T\times [-\hat t/2,\hat t/2])}$ is computed under the product metric $h'=g\llcorner T\oplus ds^2$;
\item the metrics $\sigma_\epsilon^*\he$ converge in the $C^0$ topology to $\gamma^*h\llcorner T\times [-\hat t/2,\hat t/2]$ (see \eqref{eq:define gamma} for the definition of $\gamma$).
\end{enumerate}
\end{lemma}
\begin{proof}
Given $\epsilon>0$, recall that 
\[s(t)=-\int_t^\epsilon \vartheta_\epsilon (u)du. \]
Then we define the embedding map $\sigma_\epsilon:T\times [-\hat t/2,\hat t/2]\to N$ by 
\[  \sigma_\epsilon(x,u)= \ms F(x, s^{-1}(u)). \] 
The first item follows immediately. The second one comes from the fact that $\vartheta_\epsilon\geq 1$. The last one follows from Lemmas \ref{lem:CN is C1} and \ref{lem:convergence to CN} and the fact that $g=\he$ on $N\setminus \ms F(T\times [0,\epsilon])$.	
	
	\end{proof}

\subsection{Notations from geometric measure theory}
We now recall the formulation in \cite{LMN16}. Let $(M,\partial M,g)\subset \mb R^L$ be a compact Riemannian manifold with piecewise smooth boundary. Let $\mc R_k(M;\mb Z_2)$ (resp. $\mc R_{k}(\partial M)$) be the space of $k$-dimensional rectifiable currents in $\mb R^L$ with coefficients in $\mb Z_2$ which are supported in $M$ (resp. $\partial M$). Denote by $\mf M$ the mass norm. We now recall the formulation in \cite{LZ16} using equivalence classes of integer rectifiable currents. Let 
\begin{equation}
\label{eq:def from integer rect}
 Z_k(M,\partial M;\mb Z_2):=\{T\in \mc{R}_k(M;\mb Z_2) : \spt(\partial T)\subset \partial M \}.
 \end{equation}
We say that two elements $S_1,S_2\in Z_k(M,\partial M;\mb Z_2)$ are equivalent if $S_1-S_2\in \mc{R}_k(\partial M;\mb Z_2)$. Denote by $\mc{Z}_k(M,\partial M; \mb Z_2)$ the space of all such equivalence classes. For any $\tau\in \mc{Z}_k(M,\partial M; \mb Z_2)$, we can find a unique $T\in \tau$ such that $T\llcorner \partial M=0$. We call such $S$ the \emph{canonical representative} of $\tau$ as in \cite{LZ16}. For any $\tau\in \mc{Z}_k(M,\partial M;\mb Z_2)$, its mass and flat norms are defined by 
\[\mf{M}(\tau):=\inf\{\mf{M}(S) : S\in \tau\}\quad \text{ and }\quad \mc{F}(\tau):=\inf\{\mc{F}(S) : S\in \tau\}.\]
The support of $\tau\in \mc{Z}_k(M,\partial M;\mb Z_2)$ is defined by 
\[\spt (\tau):=\bigcap_{S\in \tau}\spt (S).\]
By \cite{LZ16}*{Lemma 3.3}, we know that for any $\tau \in \mc{Z}_k(M,\partial M;\mb Z_2)$, we have $\mf{M}(S)=\mf{M}(\tau)$ and $\spt (\tau)=\spt (S)$, where $S$ is the canonical representative of $\tau$.

Recall that the varifold distance function $\mf{F}$ on $\mc{V}_k(M)$ is defined in \cite{Pi}*{2.1 (19)}, which induces the varifold weak topology on the set $\mc{V}_k(M)\cap \{V : \|V\|(M)\leq c\}$ for any $c$. We also need the $\mf{F}$-metric on $\mc{Z}_k(M,\partial M;\mb Z_2)$ defined as follows: for any $\tau, \sigma\in \mc{Z}_k(M,\partial M;\mb Z_2)$ with canonical representatives $S_1\in \tau$ and $S_2\in \sigma$, the $\mf{F}$-metric of $\tau$ and $\sigma$ is 
\[\mf{F}(\tau,\sigma):=\mc{F}(\tau-\sigma)+\mf{F}(|S_1|,|S_2|),\]
where $\mf{F}$ on the right hand side denotes the varifold distance on $\mc{V}_k(M)$. 

For any $\tau\in \mc{Z}_k(M,\partial M; \mb Z_2)$, we define $|\tau|$ to be $|S|$, where $S$ is the unique canonical representative of $\tau$ and $|S|$ is the rectifiable varifold corresponding to $S$.

We assume that $\mc{Z}_k(M,\partial M;\mb Z_2)$ has the flat topology induced by the flat metric. With the topology of mass norm or the $\mf{F}$-metric, the space will be denoted by  $\mc{Z}_k(M,\partial M;\mf{M};\mb Z_2)$ or  $\mc{Z}_k(M,\partial M;\mf{F};\mb Z_2)$.

Let $X$ be a finite dimensional simplicial complex. Suppose that $\Phi:X\to \mc{Z}_n(M,\partial M;\mf{F};\mb Z_2)$ is a continuous map with respect to the $\mf{F}$-metric. We use $\Pi$ to denote the set of all continuous maps $\Psi:X\to \mc{Z}_n(M,\partial M;\mf{F};\mb Z_2) $ such that $\Phi$ and $\Psi$ are homotopic to each other in the flat topology. The \emph{width} of $\Pi$ is defined by 
\[\mf{L}(\Pi)=\inf_{\Phi\in \Pi}\sup_{x\in X} \mf{M}(\Phi(x)).\]

Given $p\in \mb N$, a continuous map in
the flat topology
\[\Phi : X \rightarrow \mc Z_{n} (M, \partial M;\mb Z_2)\]
is called a {\em $p$-sweepout} if the $p$-th cup power of $\lambda = \Phi^*(\bar{\lambda})$ is non-zero in $H^{p}(X; \mb Z_2 )$ where $0 \neq \bar{\lambda} \in H^1(\mc Z_{n} (M, \partial M;\mb Z_2);\mb Z_2) \cong \mb Z_2$. Denote by $\mc P_p(M)$ the set of all $p$-sweepouts that are continuous in the flat topology and {\em have no concentration of mass} (\cite{MN17}*{\S 3.7}), i.e.
\[\lim_{r\rightarrow 0}\sup\{\mf M(\Phi(x)\cap B_r(q)):x\in X,q\in M\}=0.\]

In \cite{MN17} and \cite{LMN16}, the {\em $p$-width} is defined as
\begin{equation}
\omega_p(M;g):=\inf_{\Phi\in\mc P_p}\sup\{\mf M(\Phi(x)):x\in \mathrm{dmn}(\Phi)\}.
\end{equation}

\begin{remark}
In this paper, we used the integer rectifiable currents, which is the same with \cite{LZ16}. However, the formulations are equivalent to that in \cite{LMN16}; see \cite{GLWZ19}*{Proposition 3.2} for details.
\end{remark}

For the non-compact setting, the following definition does not depend on the choice of the exhaustion sequences by \cite{LMN16}*{Lemma 2.15 (1)}.
\begin{definition}[\cite{Song18}*{Definition 7}]
Let $(\hat N^{n+1},g)$ be a complete non-compact Lipschitz manifold. Let $K_1\subset K_2\subset\cdots\subset K_i\subset\cdots$ be an exhaustion of $\hat N$ by compact $(n+1)$-submanifolds with piecewise smooth boundary. The $p$-width of $(\hat N,g)$ is the number
\[\omega_p(\hat N;g)=\lim_{i\rightarrow\infty }\omega_p(K_i;g)\in[0,+\infty].\]
\end{definition}

\subsection{Min-max theory for manifolds with boundary and ends}
Let $(N,\partial N,T,g)$ be a compact manifold with boundary and portion such that $T$ is a free boundary minimal hypersurface in $(N,\partial N,g)$ with a contracting neighborhood in one side in $N$.  Let $T_1,\cdots, T_m$ be the connected components of $T$ and suppose that $T_1$ has the largest area among their components:
\[\Area(T_1)\geq \Area(T_j) \ \ \text{ for all }  \ \ j\in\{1,\cdots, m\}.\]

The purpose of this subsection is to prove the $p$-width $\omega_p(\mc C(N))$ is associated with almost properly embedded free boundary minimal hypersurfaces with multiplicities.

We give the upper and lower bounds for $\omega_p(\mc C(N);h)$.
\begin{lemma}[cf. \cite{Song18}*{Theorem 8}]\label{lem:growth of width}
There exists a constant $C$ depending on $h$ such that for all $p\in\{1,2,3,\cdots\}$:
\begin{gather}
\omega_{p+1}(\mc C(N))-\omega_p(\mc C(N))\geq \Area(T_1);\label{eq:gap of wp}\\
p\cdot \Area(T_1)\leq \omega_p(\mc C( N))\leq p\cdot\Area(T_1)+C p^{\frac{1}{n+1}}.\label{eq:upper bound of wp}
\end{gather}
\end{lemma}
\begin{proof}
The proof here actually is the same with \cite{Song18}*{Theorem 8}, which is an application of Lusternick-Schnirelman Inequalities in \cite{LMN16}*{Section 3.1}. We sketch the idea here.

Firstly, we know that $\omega_1(T\times[-R,R];h)$ is realized by a varifold $V_R$. Then by \cite{LZ16}, $V_R$ is a free boundary minimal hypersurface when restricted outside $\partial T\times\{\pm R\}$.  Moreover, by \cite{LZ16} again, the first variation of $V_R$ vanishes along each vector field $X$ satisfying that $X$ is tangential on $T\times \{\pm R\}$ and $\partial T\times (-R,R)$. Since each slice of $T\times \{t\}$ is minimal and stable, then by Lemma \ref{lem:no mass accumulate}, for any neighborhood $U$ of $\partial T\times\{\pm R\}$, the support of $V_R$ intersects $U\setminus\partial T\times\{\pm R\}$ if it intersects $\partial T\times\{\pm R\}$. Together with the monotonicity formula and maximum principle (see \cite{Song18}*{Theorem 8} for details), we always have 
\[\omega_1(T\times [-R,R])=\Area(T_1)\]
for sufficiently large $R$. Letting $R\rightarrow \infty$, we conclude that
\[\omega_1(T\times \mb  R;h)=\Area(T_1).\]
Then by Lusternick-Schnirelman Inequalities,
\[\omega_{p+1}(T_1\times [0,2R];h)\geq \omega_p(T_1\times [0,R];h)+\omega_1(T_1\times (R,2R];h).\]
Letting $R\rightarrow \infty$,
\[\omega_{p+1}(T\times \mb R;h)\geq \omega_p(T_1\times \mb R;h)+\Area(T_1).\]
By induction, $\omega_p(T_1\times \mb R)\geq p\cdot \Area(T_1)$. On the other hand, by direct construction, we have that $\omega_p(T_1\times \mb R)\leq p\cdot \Area(T_1)$. Therefore,
\[\omega_p(T_1\times \mb R)= p\cdot \Area(T_1).\]

We now prove \eqref{eq:gap of wp}. Fix $q\in N$ and take $R$ large enough so that $B(q,3R)$ contains two disjoint part $B(q,R)$ and $T_1\times [0,R]$. Then by Lusternick-Schnirelman Inequalities,
\[\omega_{p+1}(B(q,3R);h)\geq \omega_p(B(q,R);h)+\omega_1(T_1\times [0,R];h).\]
Letting $R\rightarrow 0$, then we have 
\[\omega_{p+1}(\mc C)\geq \omega_p(\mc C)+\Area(T_1),\]
which is exactly the desired inequality.

In the next, we prove \eqref{eq:upper bound of wp}. Clearly, the first half follows from \eqref{eq:gap of wp}. Using \cite{LMN16}*{Lemma 4.4} (see also \cite{Song18}*{Proof of Theorem 8}), 
\[\omega_p(\mc C)\leq \omega_p(N)+\omega_p(T\times \mb R)\leq p\cdot \Area (T_1)+C\cdot p^{\frac{1}{n+1}}.\]
Here the last inequality we used the Weyl Law of $\omega_p(N)$ by Liokumovich-Marques-Neves \cite{LMN16}*{\S 1.1}. This finishes the proof.
\end{proof}

Let $h_\epsilon$ be a the metric constructed in Subsection \ref{subsection:add cylinders}. Denote by $N_\epsilon=N\setminus \ms F(T\times[0,\epsilon/2))$, which is a compact manifold with piecewise smooth boundary $\wti \partial N_\epsilon$. For simplicity, denote by $T_\epsilon= \ms F(T\times \{\epsilon/2\})$. Although there is no general regularity for min-max theory in such a space, we can use the uniform upper bound of the width and the monotonicity formulas \citelist{\cite{GLZ16}*{Theorem 3.4}\cite{Si}*{\S 17.6}} to prove that $\omega_p(N_\epsilon;\he)$ is realized by embedded free boundary minimal hypersurfaces.
\begin{theorem}\label{thm:min-max for Nepsilon}
Fix $p\in\mb N$. For $\epsilon>0$ small enough, there exist disjoint, connected, almost properly embedded, free boundary minimal hypersurfaces $\Gamma_1,\cdots,\Gamma_N$ contained in $N_\epsilon\setminus \ms F(T\times \{\epsilon/2\})$ and positive integers $m_1,\cdots,m_N$ such that
\[\omega_p(N_\epsilon;\he)=\sum_{j=1}^Nm_j\cdot\Area(\Gamma_j)\ \ \text{ and } \  \ \sum_{j=1}^N\mathrm{Index}(\Gamma_j)\leq p.\]
\end{theorem}
\begin{proof}
Choose a sequence $\{\Phi_i\}_{i\in\mb N}\subset \mc P_p(N_\epsilon)$ such that
\begin{equation}\label{eq:sequence of sweepouts realize width}
\lim_{i\rightarrow\infty} \sup\{\mf M(\Phi_i(x)):x\in X_i = \mathrm{dmn}(\Phi_i)\}=\omega_k(N_\epsilon; g).
\end{equation}
Without loss of generality, we can assume that the dimension of $X_i$ is $p$ for all $i$ (see \cite{MN16}*{\S 1.5} or \cite{IMN17}*{Proof of Proposition 2.2}).

By the Discretization Theorem \cite{LZ16}*{Theorem 4.12} and the Interpolation Theorem \cite{GLWZ19}*{Theorem 4.4}, we can assume that $\Phi_i$ is a continuous map to $\mc Z_n(N_\epsilon,\wti \partial N_\epsilon;\mb Z_2)$ in the $\mf F$-metric. Denote by $\Pi_i$ the homotopy class of $\Phi_i$. By \cite{GLWZ19}*{Proposition 7.3, Claim 1}, 
\[\lim_{i\rightarrow\infty}\mf L(\Pi_i)=\omega_p(N_\epsilon;h_\epsilon).\]
For any $p\in\{1,2,3,\cdots\}$, by \cite{MNS17}*{Lemma 1} and Lemma \ref{lem:convergence to CN},
\[\lim_{\epsilon\rightarrow\infty}\omega_p(N_\epsilon;\he)=\omega_p(\mc C(N);h).\]
Hence we can assume $\mf L(\Pi_i)$ has a uniform upper bound not depending on $i$ or $\epsilon$.

\medskip
We first prove that $\mf L(\Pi_i)$ is realized by free boundary minimal hypersurfaces. Without loss of generality, we assume that 
\[\mf L(\Pi_i)<\omega_p(N_\epsilon;\he)+1.\]
By the work of Li-Zhou \cite{LZ16}*{Theorem 4.21}, there exists a varifold $V_\epsilon^{i}$ so that 
\begin{itemize}
\item $\mf L(\Pi_i)=\mf M(V_\epsilon^{i})$;
\item with respect to metric $\he$, $V_\epsilon^i$ is stationary in $N_\epsilon\setminus \partial T_\epsilon$ with free boundary; moreover, the first variation vanishes along each vector field $X$ satisfying that $X$ is tangential when restricted on $\partial N\cap N_\epsilon$ and $T_\epsilon$;     
\item with respect to metric $\he$, $V_\epsilon^i$ is almost minimizing in small annuli with free boundary for any $q\in N_\epsilon \setminus \partial T_\epsilon$.
\end{itemize}
Denote by $S_\epsilon^i$ the support of $\|V_\epsilon^i\|$. Also, by the regularity theorem given by Li-Zhou \cite{LZ16}*{Theorem 5.2}, when restricted in $N_\epsilon\setminus \partial T_\epsilon$, $S_\epsilon^i$ is a free boundary minimal hypersurface. 

Now we are going to prove that $S_\epsilon^i$ does not intersect $\partial T_\epsilon$ for $\epsilon$ small enough. Suppose not, then by Lemma \ref{lem:no mass accumulate}, for any neighborhood $U$ of $\partial T_\epsilon$ in $N$, $S_\epsilon^i$ intersects $U\setminus \partial T_\epsilon$. $S_\epsilon^i$ intersects $U\setminus \partial T_\epsilon$ for all neighborhood of $T_\epsilon$.
By the maximum principle, $\Si_\epsilon^i$ also has to intersect $\ms F(T\times\{\hat t\})$. Note that $\mf M(V_\epsilon^i)$ is uniformly bounded from above for $i$ since $\mf L(\Pi_i)$ is uniformly bounded. This contradicts the monotonicity formula \citelist{\cite{GLZ16}*{Theorem 3.4}\cite{Si}*{\S 17.6}}. Hence $S_\epsilon^i$ is an almost properly embedded free boundary minimal hypersurface in $N_\epsilon\setminus T_\epsilon$.

\medskip
Next we prove the index bound for $S_\epsilon^i$. Such a bound follows from the argument in \cite{MN16} (see also \cite{GLWZ19}*{Theorem 6.1} for free boundary minimal hypersurfaces) if we can construct a  sequence of metrics $h_\epsilon^j\rightarrow \he$ in the $C^\infty$ topology on $N$ so that all the free boundary minimal hypersurface in $(N,\partial N,T,h_\epsilon^j)$ is countable.

To do this, we first embed $(N,\partial N,T,\he)$ isometrically into a compact manifold with boundary $(\hat N,\partial \hat N,g_\epsilon)$. By \cite{ACS17}, we can get a sequence of smooth metrics $h_\epsilon^j\rightarrow g_\epsilon$ on $\hat N$ so that every finite cover of free boundary minimal hypersurface in $(\hat N,\partial\hat  N,h_\epsilon^j)$ is non-degenerate. Then using the argument in \cite{GLWZ19}*{Proposition 5.3} (see also \cite{Wang19}), the free boundary minimal hypersurfaces in $(N,\partial N,T,h_\epsilon^j)$ is countable.

\medskip
Now we have proved that for $\epsilon$ small enough, there exists $V_\epsilon^i$ so that $\mf L(V_\epsilon^i)=\mf L(\Pi_i)$ and the support of $V_\epsilon^i$ is a free boundary minimal hypersurface $S_\epsilon^i$ with $\mathrm{Index}(S_\epsilon^i)\leq p$. Letting $i\rightarrow\infty$, this theorem follows from the compactness for free boundary minimal hypersurfaces in \cite{GWZ18}.
\end{proof}

\begin{remark}\label{rmk:support in a large ball}
Furthermore, the monotonicity formulas \citelist{\cite{GLZ16}*{Theorem 3.4}\cite{Si}*{\S 17.6}} and mean convex foliation also indicate that there is $R>0$ and a point $q_0\in N\setminus \ms F(T\times [0,\hat t\,])$ such that for all $\epsilon$ small enough, $S_\epsilon^i$ is contained in the ball $B_\he(q_0,R)$. 
\end{remark}

Now we can prove the main result in this section, which can been seen as an analog of \cite{Song18}*{Theorem 9}.
\begin{theorem}\label{thm:min-max for cmpt with ends}
Let $(N,\partial N,T,g)$ be a compact manifold with boundary and portion in Theorem \ref{thm:min-max for Nepsilon}. Let $(\mc C(N),h)$ be as in Subsection \ref{subsection:add cylinders}. For all $p\in\{1,2,3,\cdots\}$, there exist disjoint, connected, embedded free boundary minimal hypersurfaces $\Gamma_1,\cdots,\Gamma_N$ contained in $N\setminus T$ and positive integers $m_1,\cdots,m_N$ such that 
\[\omega_p(\mc C(N);h)=\sum_{j=1}^Nm_j\Area(\Gamma_j).\]
Besides, if $\Gamma_j$ is one-sided then the corresponding multiplicity $m_j$ is even.
\end{theorem}
\begin{proof}
We follow the steps given by Song in \cite{Song18}.

Recall that $N_\epsilon=N\setminus \ms F(T\times [0,\epsilon/2))$. By Theorem \ref{thm:min-max for Nepsilon} and Remark \ref{rmk:support in a large ball}, we obtain a varifold $V_\epsilon$ so that
\begin{itemize}
\item $\mf M(V_\epsilon)=\omega_p(N_\epsilon;\he)$;
\item  the support of $V_\epsilon$ is an almost properly embedded free boundary minimal hypersurface, denoted by $S_\epsilon$; 
\item for fixed $p>0$, there exist $R>0$ and a point $q_0\in N\setminus \ms F(T\times [0,\hat t\,])$ such that for all $\epsilon$ small enough, $S_\epsilon^i$ is contained in the ball $B_\he(q_0,R)$;
\item $\mathrm{Index}(\text{support of }V_\epsilon)\leq p$.
\end{itemize}

The next step is to take a limit as a sequence $\epsilon_k\rightarrow 0$. Note that $\omega_p(N_\epsilon;h_\epsilon)$ converges to $\omega_p(\mc C(N);h)$. Thus $V_{\epsilon_k}$ subsequently converges to a varifold $V_\infty$ in $\mc C(N)$ of total mass $\omega_p(\mc C(N);h)$, whose support is denoted by $S_\infty$.

Using the compactness again, $S_\infty\llcorner(\mc C(N)\setminus T)$ is an almost properly embedded free boundary minimal hypersurface since $h_\epsilon$ converges smoothly in this region. Then by the maximum principle again, $S_\infty$ is contained in the compact set $(N,g)$. Furthermore, we will prove that $V_\infty$ is $g$-stationary with free boundary on $\partial N$. Once this has been proven, then applying \cite{Song18}*{Proposition 3}, $V_\infty$ is actually a $g$-stationary integral varifold with free boundary. Recall that each connected component intersects $F(T\times \{\hat t\})$. Hence no component of $S_\infty$ is contained in $T$. Then by the strong maximum principle in Lemma \ref{lem:strong MP}, $S_\infty\subset N\setminus T$. Therefore, from the compactness \cite{GWZ18}, $S_\infty$ is a free boundary minimal hypersurface in $N\setminus T$, and we also conclude that the one-sided components of $S_\infty$ have even multiplicities.

\medskip
It remains to show that $V_\infty$ is $g$-stationary with free boundary in $(N,\partial N,T,g)$. For $\epsilon\geq0$, we will denote by $\nabla^\epsilon$ and $\dv^\epsilon$ the connection and divergence computed in the metric $h_\epsilon$ (by convention $h_0=g$). Let $\mathfrak X(N,\partial N)$ be the collection of vector fields $X$ so that
\begin{itemize}
\item $X(x)\in T_xN$ for any $x\in N$;
\item $X$ can be extended to a smooth vector field on $\wti N$;
\item $X(x)\in T_x(\partial N)$ for any $x\in\partial N$;
\end{itemize}
Our goal is to prove that the first variation along $X\in\mathfrak X(N,\partial N)$ vanishes:
\begin{equation}\label{eq:aim:vanish 1st variation}
\delta V_\infty(X)=\int\dv_S^0 X(x)dV_\infty(x,S)=0.
\end{equation}
We use the same strategy with \cite{Song18}*{Proof of Theorem 8}. In the following, we give the necessary modification and put the computation in Appendix \ref{sec:app:Vinfty is stationary}.

\medskip
{\noindent\bf Part I:} {\em Normalize the coordinate function  with respect to $\he$.}

Recall that for $\epsilon>0$ small enough, the map
\[
\ms F:T\times [0,\hat t\,]\rightarrow N
\]
is a diffeomorphism onto its image. Note that the support of $V_\infty$ restricted to $N\setminus T$ is an almost properly embedded free boundary minimal hypersurface. Hence we can assume that the vector field $X$ is supported in $\ms F(T\times [0,\hat t/2\,])$. Thus for all $\epsilon$ small enough, the vector field $X$ restricted to $N_\epsilon:=N\setminus \ms F(T\times(0,\epsilon/2))$ can be decomposed into two components 
\[X=X^\epsilon_\perp+X^\epsilon_\parallel,\]
where $X^\epsilon_\perp$ is orthogonal to $\nabla^\epsilon t$ and $X^\epsilon_\parallel$ is a multiple of $\nabla^\epsilon t$. 

For $q=\ms F(x,t)$, denote
\[\n(q):=f(q)\vte(t)\nabla^\epsilon t,\]
which is a unit vector field with respect to the metric $h_\epsilon$. Recall that the coordinate $s$ is defined by
\[s(\ms F(x,t)):=-\int_t^\epsilon\vte(u)du.\]
Then for the points where the metric is changed, $s$ is negative. Clearly,
\[\nabla^\epsilon s=\vartheta_\epsilon(t)\nabla^\epsilon t=(\vartheta_\epsilon(t))^{-1}\nabla t,\]
which implies that
\[|\nabla^\epsilon s|_{\he}=(f(q))^{-1}=|\nabla t|_{g}.\]
We use $\pps $ and $\ppt$ to denote $\ms F_*(\pps)$ and $\ms F_*(\ppt)$, respectively. Then we also have
\[  \pps=(\vartheta_\epsilon(t))^{-1}\ppt .\]
Recall that the map $F$ is defined by the first eigenfunction in Lemma \ref{lemma:good neighborhood} and $\nabla t|_T=\phi^{-1}\n$. Then we can normalize and fix such a positive function so that $\max_{\{x\in T\}}\phi=1$. Since $\nabla t$ is a smooth vector field, then for $\epsilon$ small enough, 
\begin{equation}\label{eq:bound f}
2\max_{x\in T}\phi^{-1}\geq |\nabla t|_g\geq 1/2,\ \text{ for }\ x\in\ms F(T\times[0,2\epsilon]).
\end{equation}
Let $(\gamma(u))_{0\leq u\leq r}$ be a geodesic in $(N_\epsilon,h_\epsilon)$ with $\gamma(0)\in \ms F(T\times\{\epsilon\})$. Then
\[
s(\gamma(r))-s(\gamma(0))=\int_0^r\he(\nabla^\epsilon s,\gamma'(u))du\geq -\int_0^r|\nabla^\epsilon s |_{\he}du\geq -2r\max_{x\in T}\phi^{-1}.
\]
If we take $C_0=2\max_{x\in T}\phi^{-1}$, then 
\begin{equation}\label{eq:d is equiv to s}
B_{\he}(q_0,R)\subset \Big[N\setminus\ms  F(T\times[0,\epsilon])\Big]\cup\{q\in\ms  F(T\times [0,\epsilon]):s\geq -C_0R\}.
\end{equation}

\medskip
{\noindent\bf Part II:} {\em The uniform upper bound for points with non-parallel normal vector field.}

Let $(y,S)$ be a point of the Grassmannian bundle of $N$ and let $(e_1,\cdots,e_n)$ be an $h_\epsilon$-orthonormal basis of $S$ so that $e_1,\cdots,e_{n-1}$ are $h_\epsilon$-orthogonal to $\nabla^\epsilon t$. Denote by $\bar \n$ the unit normal vector of $S$ under the metric $h_\epsilon$. Let $e^*_n$ be a unit vector such that $(e_1,\cdots,e^*_n)$ is an $h_\epsilon$-orthonormal basis of the $n$-plane $h_\epsilon$-orthogonal to $\nabla^\epsilon t$ at $y$.

The main result in this part is that for any $b>0$,
\begin{equation}\label{eq:small bad part}
\lim_{\epsilon\rightarrow 0}\int_{\ms F(T\times[0,2\epsilon])\times \mf G(n+1,n)}\chi_{\{|\he(e_n,\n)|>b\}}dV_\epsilon(x,S)=0.
\end{equation}
In particular, 
\begin{equation}\label{eq:parallel to T}
V_\infty\llcorner\{(x,S):x\in T,S\neq T_x T\}=0.
\end{equation}

The proof is similar to Song \cite{Song18}*{(11)}. We postpone the proof of \eqref{eq:small bad part} to Subsection \ref{subsec:app:bad part is small} in Appendix \ref{sec:app:Vinfty is stationary}.

\medskip
We now explain how to deduce \eqref{eq:aim:vanish 1st variation} from the previous estimates. Take a sequence $\epsilon_k\rightarrow 0$. Consider
\[
A_k:=\ms F(T\times[0,2\epsilon_k])\ \  \text{ and }\ \  B_k:=N\setminus\ms F(T\times [0,2\epsilon_k]).
\]
Then by taking a subsequence (still denoted by $A_k$ and $B_k$), we can assume that there are two varifolds $V'_\infty$ and $V''_\infty$ in $N$ so that as $k\rightarrow\infty$, the following convergences in the varifolds sense take place:
\begin{gather*}
V_k:=V_{\epsilon_k}\rightharpoonup V_\infty,\\
V_k':=V_{\epsilon_k}\llcorner(A_k\times \mf G(n+1,n))\rightharpoonup V_\infty',\\
V_k'':=V_{\epsilon_k}\llcorner(B_k\times \mf G(n+1,n))\rightharpoonup V_\infty''.
\end{gather*}
Recall that we decomposed $X=\Xoe+\Xpe$. 

\medskip
{\noindent\bf Part III:} {\em We will show first that 
\begin{equation}\label{eq:normal part}
\int \dv^0X_\perp^0dV_\infty= \lim_{k\rightarrow\infty}\int\dv^{\epsilon_k}X_\perp^{\epsilon_k}dV_k=0.
\end{equation}}

Let $(x,S)$ and $e_1,\cdots,e_n,e_n^*$ be defined as before and let $S_\perp$ denote the $n$-plane at $x$ orthogonal to $\nabla s$. By the construction of $h_\epsilon$, we have that for any $e'\in S_\perp$,
\begin{equation}
\he(\nabla^\epsilon_{e'}\Xoe,e')=g(\nabla^0_{e'}\Xoe,e').
\end{equation}
Then a direct computation gives that 
\[\dve_S\Xoe=\dv^0_{S_\perp}\Xoe+\Upsilon(\epsilon,x,S,X),\]
where
\begin{align}\label{eq:bound Upsilon}
\Upsilon(\epsilon,x,S,X)&=\he(\nabla^\epsilon_{e_n}\Xoe,e_n)-\he(\nabla^\epsilon_{e_n^*}\Xoe,e_n^*)\\
&\leq 2|\nabla^\epsilon\Xoe|_{\he}\cdot|e_n-e_n^*|_{\he}.\nonumber
\end{align}
By the construction of $\he$, we have that $|\nabla^\epsilon\Xoe|_{\he}$ is uniformly bounded in $\epsilon>0$. Together with \eqref{eq:small bad part}, we in fact have (see Subsection \ref{subsec:proof of normal V'} for details)
\begin{equation}\label{eq:normal V'}
\lim_{k\rightarrow\infty}\int\dv_S^{\epsilon_k}X_\perp^{\epsilon_k}dV_k'(x,S)=\int\dv^0X_\perp^0dV'_\infty.
\end{equation}
On the other hand, using the facts that $\he=g$ and $X^\epsilon_\perp$ smoothly converges to $X^0_\perp$ in $B_k$, we have
\[\int\dv^{0}X_\perp ^0dV_\infty''=\lim_{k\rightarrow\infty}\int\dv^{0}X_\perp^0dV_k''=\lim_{k\rightarrow\infty}\int\dv^{\epsilon_k}X^0_\perp dV_k''=\lim_{k\rightarrow\infty}\int\dv^{\epsilon_k}X^{\epsilon_k}_\perp dV_k''.\]
Then \eqref{eq:normal part} follows immediately.

\medskip
{\noindent\bf Part IV:} {\em Finally, we prove that 
\[\int \dv^0 X_\parallel^0dV_\infty =0.\]
}

 By the definition of $\Xpe$, there exists $\varphi$ so that $X^0_\parallel=\varphi \nabla t$. Now define 
\[Z^\epsilon:=\varphi\nabla^\epsilon s.\]
Then the most important thing is that $|\nabla^\epsilon Z^\epsilon|_{\he}$ is uniformly bounded (see Subsection \ref{subsec:new vector has bounded gradient}). Using the same argument in \cite{Song18}*{Theorem 9}, such a property enables us (see Subsection \ref{subsec:parallel V''}) to prove that
\begin{equation}\label{eq:parallel V''}
\lim_{k\rightarrow\infty}\Big|\int\dv_S^{\epsilon_k}X_\parallel^{\epsilon_k}dV_k''(x,S)\Big|=0.
\end{equation}
Using the facts that $\he=g$ and $X^\epsilon_\parallel$ smoothly converges to $X^0_\parallel$ in $B_k$, we have
\[\int\dv^{0}X_\parallel^0dV_\infty''=\lim_{k\rightarrow\infty}\int\dv^{0}X_\parallel^0dV_k''=\lim_{k\rightarrow\infty}\int\dv^{\epsilon_k}X^0_\parallel dV_k''=\lim_{k\rightarrow\infty}\int\dv^{\epsilon_k}X^{\epsilon_k}_\parallel dV_k''=0.\]
On the other hand, the minimality of $T$ and \eqref{eq:parallel to T} give that 
\[\int\dv^0 X_\parallel^0dV_\infty'=0.\]
Therefore,
\[\int \dv^0 X_\parallel^0dV_\infty =\int\dv^0 X_\parallel^0dV_\infty'+\int\dv^{0}X_\parallel^0dV_\infty''=0.\]
The desired equality \eqref{eq:aim:vanish 1st variation} follows from Part III and IV.
\end{proof}

\section{Proof of main theorem}\label{sec:proof of main thm}
Now we are ready to prove our main theorem. The conditions  (\ref{assump:contracting portion}--\ref{assump:at most one minimal component}) defined in Subsection \ref{subsec:constr of area minimizer} will be used frequently.
\begin{theorem}\label{thm:infitely many fbmhs}
Let $(M^{n+1},\partial M,g)$ be a connected compact Riemannian manifold with smooth boundary and $3\leq (n+1)\leq 7$. Then there exist infinitely many almost properly embedded free boundary minimal hypersurfaces.
\end{theorem}
\begin{proof}
Assume on the contrary that $(M,\partial M,g)$ contains only finitely many free boundary minimal hypersurfaces. Then by the construction in Lemma \ref{lemma:good neighborhood}, \eqref{assump:regular neighborhood} and \eqref{assump:half regular neighborhood} hold true.

\medskip
Now we prove that by cutting along free boundary minimal hypersurfaces in finite steps, we can construct a compact manifold with boundary and portion satisfying Frankel property and each free boundary minimal hypersurface that does note intersect the portion must have area larger than each connected component of the portion.

\medskip
Let $T_0^0$ be the union of the connected components of $\partial M$ which is a  closed minimal hypersurface having a contracting neighborhood in one side in $M$. Denote by $M_0^0:=M$ and $\partial M_0^0=\partial M\setminus T_0^0$. Then $(M^0_0,\partial M_0^0,T_0^0,g)$ is a compact manifold with boundary and portion satisfying \eqref{assump:contracting portion}, \eqref{assump:regular neighborhood} and \eqref{assump:half regular neighborhood}. 

\medskip
Firstly, cut $M_0^0$ along a one-sided properly embedded free boundary minimal hypersurface $\Gamma_0$ of $(M_0^0,\partial M_0^0,T_0^0,g)$ in $M^0_0\setminus T_0^0$ having a contracting neighborhood. Denote by $M_1^0$ the closure of $M_0^0\setminus \Gamma_0$ and define
\[\partial M^0_1:=M^0_1\cap \partial M_0^0 \ \ \text{ and } \ \ T_1^0:=T_0^0\cup \wti \Gamma_0,\]
where $\wti \Gamma_0$ is the double cover of $\Gamma_0$ in $M^0_1$. Then repeat this procedure by cutting $M_1^0$ along a one-sided free boundary minimal hypersurface $\Gamma_1\subset M_1^0\setminus \wti \Gamma_0$. Thus we construct a finite sequence $(M_0^0,\partial M_0^0,T_0^0,g)$, $(M_1^0,\partial M_1^0,T_1^0,g),\cdots,(M_J^0,\partial M_J^0,T_J^0,g)$ by successive cuts. Then after finitely many times (denoted by $J$), $M_J^0\setminus T_J^0$ does not contain any one-sided properly embedded free boundary minimal hypersurfaces having a contracting neighborhood. Denote by 
\[(M^1_0,\partial M^1_0,T^1_0,g)=(M_J^0,\partial M_J^0,T_J^0,g).\]
Clearly, $(M^1_0,\partial M^1_0,T^1_0,g)$ satisfies \eqref{assump:contracting portion}, \eqref{assump:regular neighborhood}, \eqref{assump:half regular neighborhood} and \eqref{assump:one sided expande}.

\medskip
Secondly, we cut $M^1_0$ along a two-sided, properly embedded, free boundary minimal hypersurface $\Gamma_0'$ in $(M^1_0,\partial M^1_0,T^1_0,g)$ that has a contracting neighborhood. Denote by $M_1^1$ the closure of one of the connected components of $M_0^1\setminus \Gamma_0'$ and define
\[\partial M_1^1:=M_1^1\cap \partial M_0^1 \ \ \text{ and } \ \ T_1^1:=M^1_1\cap (T_0^1\cup \Gamma_{0,1}'\cup \Gamma_{0,2}'),\]
where $\Gamma_{0,1}'$ and $\Gamma_{0,2}'$ are the two free boundary minimal hypersurfaces that are both isometric to $\Gamma_0'$. Then after finitely many times, we obtain a compact manifold with boundary and portion (denoted by $(M^2_0,\partial M^2_0,T^2_0,g)$) that every properly embedded free boundary minimal hypersurface in $M_0^2\setminus T_0^2$ has an expanding neighborhood. Moreover, we have that:
\begin{claim}\label{claim:two sides generic separation}
Every two-sided properly embedded free boundary hypersurface of $(M^2_0,\partial M^2_0,T^2_0,g)$ in $M^2_0\setminus T_0^2$ separates $M^2_0$.
\end{claim}
\begin{proof}[Proof of Claim \ref{claim:two sides generic separation}]
If not, there is a two-sided free boundary hypersurface $\Sigma$ in $(M^2_0,\partial M^2_0,T^2_0,g)$ does not separate $M_0^2$. Then $\Sigma$ represents a nontrivial relative homology class in $(M^2_0,\partial M^2_0)$. Then we can obtain an area minimizer, which contains a component $S$ in $M_0^2\setminus T_0^2$. In particular, $S$ is properly embedded and has a contracting neighborhood, which contradicts \eqref{assump:one sided expande} and the fact that every properly embedded free boundary minimal hypersurface in $(M^2_0,\partial M^2_0,T^2_0,g)$ has an expanding neighborhood.
\end{proof}

Similarly, we have the following:
\begin{claim}\label{claim:only one minimal}
At most one connected component of $\partial M^2_0$ is a closed minimal hypersurface, and if it happens, it has an expanding neighborhood in one side in $M^2_0$.
\end{claim}
\begin{proof}[Proof of Claim \ref{claim:only one minimal}]
We argue by contradiction. Assume there are two disjoint connected components $\Gamma''_1$ and $\Gamma_2''$ in $\partial M^2_0$ are closed minimal hypersurfaces. Then by the definition of $T_0^0$, both $\Gamma''_1$ and $\Gamma_2''$ have expanding neighborhoods in one side in $M^2_0$. Then $\Gamma_1''$ represents non-trivial relative homology class in $(M^2_0,\partial M^2_0\setminus (\Gamma_1''\cup\Gamma_2''))$. By minimizing the area of this class, we obtain a properly embedded free boundary minimal hypersurface having a contracting neighborhood, which leads to a contradiction.
\end{proof}

Claim \ref{claim:two sides generic separation} gives that each two-sided free boundary minimal hypersurface generically separates $M^2_0$ (see Subsection \ref{subsec:constr of area minimizer}). Claim \ref{claim:only one minimal} implies that $(M^2_0,\partial M^2_0,T^2_0,g)$ satisfies \eqref{assump:at most one minimal component}. Therefore, $(M^2_0,\partial M^2_0,T^2_0,g)$ satisfies (\ref{assump:contracting portion}--\ref{assump:at most one minimal component}).

\medskip
Thirdly, we cut $(M_0^2,\partial M_0^2,T_0^2,g)$ along a two-sided, half-properly embedded free boundary minimal hypersurface $\Gamma'''\subset M_0^2\setminus T_0^2$ which has a proper and contracting neighborhood in one side. By Claim \ref{claim:two sides generic separation}, $\Gamma'''$ generically separates $M_0^2$. Denote by $M_1^2$ the closure of the generic component containing the proper neighborhood in one side. Define
\[\partial M_1^2:=\overline{(M_1^2\cap \partial M_0^2)\setminus \Gamma'''} \ \ \text{ and } \ \ T_1^2=(T_0^2\cap M_0^2)\cup \Gamma'''.\]
Then $(M_1^2,\partial M_1^2,T_1^2,g)$ is a compact manifold with boundary and portion (see Figure \ref{fig:half_proper}).
\begin{figure}[h]
\begin{center}
\def\svgwidth{0.7\columnwidth}
  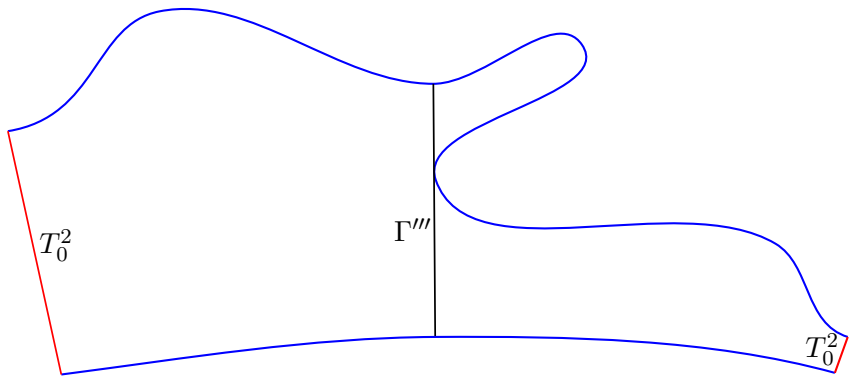
  \caption{Cutting half-properly embedded hypersurfaces.}
  \label{fig:half_proper}
\end{center}
\end{figure}
By successive cuts in finitely many times, we obtain a compact manifold with boundary and portion (denoted by $(N,\partial N,T,g)$) so that each two-sided, half-properly embedded, free boundary minimal hypersurface has a proper and expanding neighborhood in one side. By Lemma \ref{lem:general existence of good fbmh}, every two almost properly embedded free boundary minimal hypersurfaces of $(N,\partial N,T,g)$  in $N\setminus T$ intersect with each other. Without loss of generality, let $T_1$ be the connected component of $T$ so that
\[\Area(T_1)=\max\{\Area(T'):T' \text{ is a connected component of $T$}\}.\]
Then by Lemma \ref{lem:area lower bound}, each free boundary minimal hypersurface $\Sigma$ in $(N,\partial N,T,g)$ satisfies that
\begin{itemize}
\item  if $\Sigma$ is two-sided, $\Area(\Sigma)>\Area(T_1)$;
\item  if $\Sigma$ is one-sided, $2\Area(\Sigma)>\Area(T_1)$.
\end{itemize}
Thus we get the desired compact manifold with boundary and portion.

\medskip
We now proceed the proof of Theorem \ref{thm:infitely many fbmhs}. Let $\mc C(N)$ be the construction in Subsection \ref{subsection:add cylinders}. Theorem \ref{thm:min-max for cmpt with ends} gives that $\omega_p(\mc C(N);h)$ is realized by free boundary minimal hypersurfaces in $N\setminus T$. Moreover, since every two free boundary minimal hypersurfaces of $(N,\partial N,T,g)$ in $N\setminus T$ intersect each other, then there exist integers $\{m_p\}$ and free boundary minimal hypersurfaces $\{\Sigma_p\}$ so that 
\begin{equation}\label{eq:one component}
\omega_p(\mc C(N))=m_p\cdot \Area(\Sigma_p).
\end{equation}
By Lemma \ref{lem:growth of width}, the width of $\mc C(N)$ satisfies
\begin{gather*}
\omega_{p+1}(\mc C(N))-\omega_p(\mc C(N))\geq \Area(T_1);\\
p\cdot \Area(T_1)\leq \omega_p(\mc C( N))\leq p\cdot\Area(T_1)+C p^{\frac{1}{n+1}}.
\end{gather*} 
Together with \eqref{eq:one component}, we get a contradiction to \cite{Song18}*{Lemma 13}.
\end{proof}

\appendix
\section{A strong maximum principle}\label{sec:app:MP}
In \cite{Whi10}*{Theorem 4}, White gave a strong maximum principle for varifolds in closed Riemannian manifolds. Using the same spirit, Li-Zhou proved a maximum principle in compact manifolds with boundary, which played an important role in their regularity theorem for min-max minimal hypersurfaces with free boundary in \cite{LZ16}. In this appendix, we give a strong maximum principle, which is used in Theorem \ref{thm:min-max for cmpt with ends}.
\begin{lemma}[cf. \citelist{\cite{Whi10}*{Theorem 4}\cite{LZ17}*{Theorem 1.4}}]\label{lem:strong MP}
Let $(N,\partial N,T,g)$ be a compact manifold with boundary and portion so that $T$ is a free boundary minimal hypersurface. Let $V$ be a $g$-stationary varifold with free boundary in $\partial N$, i.e. for any $X\in\mathfrak X(N,\partial N)$,
\[\delta V(X)\Big(:=\int \dv XdV\Big)=0.\]
\begin{enumerate}
\item\label{item:contains whole component} If the support of $V$ (denoted by $S$) contains any point of a connected component of $T$, then $S$ contains the whole connected component;
\item\label{item:decomposition} If $V$ is a $g$-stationary integral varifold with free boundary, then $V$ can be written as $W + W'$, where the support of $W$ is the union of several connected components of $T$ and the support of $W'$ is disjoint from $T$.
\end{enumerate}
\end{lemma}
\begin{proof}
Without loss of generality, we assume that $T$ is connected and non-degenerate. We first prove \eqref{item:contains whole component} by contradiction. Assume that $S$ does not contain $T$. By \cite{SW89}*{Theorem}, $S$ does not intersect the interior of $T$. We now prove that $S\cap \partial T=\emptyset$. 

In this lemma, we always embed $N$ isometrically into a smooth, compact $(n+1)$-Riemannian manifold with boundary $(M,\partial M,g)$. We also fix a diffeomorphism $\Phi: T\times (-\delta,\delta)\rightarrow M$ which is associated with an extension of $\n$ in $\mathfrak X(N,\partial N)$. Here $\n$ is the unit outward normal vector field of $T$ in $N$.

We argue by contradiction. Assume that $p\in S\cap\partial T$. Firstly, we use \cite{SW89}*{Theorem, Step A} to construct a free boundary hypersurface outside $S$ near $p$ so that it has mean curvature vector field pointing towards $S$. To do this, we take $U\subset T$ be the neighborhood of $p$ from Proposition \ref{prop:h foliation with boundary} and $w|_{\Gamma_2}=\theta \eta$, where $\eta$ is a non-trivial and non-positive function supported in the interior of $\Gamma_2$ and $\theta>0$ is a constant. Note that $\Gamma_2 =\mathrm{Closure}(\partial U\cap \mathrm{Int} T)$. Note that $S$ does not intersect the interior of $T$. Then we can take $\theta>0$ sufficiently small so that if $\Phi(x,y)\in S$, then $y\leq \theta\eta(x)$. Fix this value $\theta$.

For simplicity, denote by $v_{s,t}$ the constructed graph function $v_t$ for $h=s$ and $w|_{\Gamma_2}=\theta \eta$ in Proposition \ref{prop:h foliation with boundary}. Then by the maximum principle, $v_{0,0}(p)<0$. Hence for $s>0$ small enough, we always have $v_{s,0}(p)<0$. Fix such $s$. Let $t_0$ be the largest $t$ so that $v_{s,t}$ intersects $S$. It follows that $t_0>0$, which implies that $S$ does not intersect $\Phi(\Gamma_2,\theta \eta+t_0)$. 

We now proceed our argument. Note that $v_{s,t_0}$ is a graph function of a free boundary hypersurface with mean curvature vector pointing towards $T$. Then by the strong maximum principle \cite{Whi10}, $S$ can not touch the interior of $\Phi ( U,v_{s,t_0})$. Using the free boundary version maximum principle \cite{LZ17}, $S$ can not touch $\Phi(\partial T\cap U,v_{s,t_0})$. Then this contradicts the construction of $v_{s,t_0}$.

\medskip
Now \eqref{item:decomposition} follows from \eqref{item:contains whole component} and a standard argument in \cite{Whi10}*{Theorem 4}. Indeed, set 
\[d:=\inf\{\{\Theta(x,V):x\in \mathrm{Int} T\}\cup\{2\Theta(x,V):x\in \partial T\}\}.\]
Then $V-d[T]$ is still a $g$-stationary integral varifold with free boundary, where $[T]$ is the the varifold associated to $T$. Then $V-d[T]$ does not contain $T$. Hence it does not intersect $T$. The proof is finished.
\end{proof}

\begin{proposition}\label{prop:h foliation with boundary}
Let $(M^{n+1},\partial M,g)$ be a compact Riemannian manifold with boundary, and let
$(\Sigma,\partial \Sigma)\subset (M,\partial M)$ be an embedded, free boundary minimal hypersurface. Given a point $p\in\partial \Sigma$, there exist $\epsilon>0$ and a neighborhood $U\subset M$ of $p$ such that if $h:U\rightarrow \mb R$ is a smooth function with $\|h\|_{C^{2,\alpha}}<\epsilon$ and 
\[w:\Sigma\cap U\rightarrow \mb R \text{ satisfies } \|w\|_{C^{2,\alpha}}<\epsilon,\]
then for any $t\in(-\epsilon,\epsilon)$, there exists a $C^{2,\alpha}$-function $v_t: U\cap \Sigma\rightarrow \mb R$, whose graph $G_t$ meets $\partial M$ orthogonally along $U\cap \partial \Sigma$ and satisfies:
\[H_{G_t} = h|_{G_t} ,\]
(where $H_{G_t}$ is evaluated with respect to the upward pointing normal of $G_t$), and
\[v_t(x) = w(x) + t, \text{ if } x\in \partial (U\cap \Sigma)\cap \mathrm{Int} M.\]
Furthermore, $v_t$ depends on $t,h,w$ in $C^1$ and the graphs $\{G_t : t\in[-\epsilon, \epsilon]\}$ forms a foliation.
\end{proposition}
\begin{proof}
The proof follows from \cite{Whi87}*{Appendix} together with the free boundary version \cite{ACS17}*{Section 3}. The only modification is that we need to use the following map to replace $\Phi$ in \cite{ACS17}*{Section 3}:
\[\Psi: \mb R \times X \times Y\times Y \times Y \rightarrow Z_1 \times Z_2 \times Z_3.\]
The map $\Psi$ is defined by
\[\Psi(t,g,h,w,u) = (H_{g(t+ w + u)}-h, g(N_g(t + w + u), \nu_g (t + w + u)), u|_{\Gamma_2} );\]
here all the notions are the same as \cite{ACS17}*{Section 3}. We remark that $\Gamma_2=\mathrm{Closure}(\partial (U\cap \Sigma)\cap \mathrm{Int} M)$.
\end{proof}

\section{Computation in the proof of Theorem \ref{thm:min-max for cmpt with ends} }\label{sec:app:Vinfty is stationary}
In this appendix, we collect the computation in Theorem \ref{thm:min-max for cmpt with ends}.
\subsection{Proof of \eqref{eq:small bad part}}\label{subsec:app:bad part is small}
Let $\varphi:\mb R\rightarrow\mb R$ be a non-negative function. Then it can also be seen as a function on $M$ by 
\[\varphi(\ms F(x,t)):=\varphi(s(\ms F(x,t))).\]
Let $H^\epsilon$ (resp $A^\epsilon$) denote the mean curvature (resp. second fundamental form) at $y$ of $\ms F(T\times \{t\})$. Let $\n:=\nabla^\epsilon s/|\nabla^\epsilon s|_{\he}=\nabla t/|\nabla t|_{\he}$. Then we have 
\[  \pps= f\n,  \]
where $\pps=\ms F_*(\pps)$. We can compute the divergence as follows:
\begin{align}
&\ \ \ \ \dve_S (\varphi \pps)=\dve_M(\varphi \pps)-\he(\nabla^\epsilon_{\bar\n} (\varphi f\n),\bar{\n})\label{eq:general divergence}\\
&=\varphi'(s)|\he(e_n,\n)|^2+\varphi \he(\nabla^\epsilon f,\n)+\varphi H^\epsilon f-\varphi \he(\nabla^\epsilon f,\bar\n)\he(\n,\bar\n)-\varphi f\he(\nabla_{\bar{\n}}
 \n,\bar{\n})\nonumber\\
&=\varphi'(s)\cdot |\he(e_n,\n)|^2+\varphi \he(\nabla^\epsilon f,\n)+\varphi H^\epsilon f-\varphi \he(\nabla^\epsilon f,\bar\n)\he(\n,\bar\n)- \varphi f\he(\nabla^\epsilon_{\n}\n,\bar{\n})\he(\n,\bar{\n})\nonumber\\
&\ \ \ \ -\varphi f\he(\nabla_{e_n^*}\n,e_n^*)\cdot|\he(\bar{\n},e_n^*)|^2\nonumber\\
&=[\varphi'(s)-\varphi fA^\epsilon(e_n^*,e_n^*)]\cdot  |\he(e_n,\n)|^2+\varphi H^\epsilon f-\varphi \he(\nabla^\epsilon f+f\nabla^\epsilon_\n\n,e_n^*)\he(\bar\n,e_n^*)\he(\n,\bar\n)+\nonumber\\
&\ \ \ \ +\varphi \he(\nabla^\epsilon f,\n)\cdot  |\he(e_n,\n)|^2.\nonumber
\end{align} 
Note that by  $\pps=f\n$,
\begin{equation}\label{eq:decom after change}
\he(\nabla_{\pps}^\epsilon\n,e_n^*)=-\he(\n,\nabla^\epsilon_{e_n^*}\pps)=-\he(\nabla^\epsilon f,e_n^*),
\end{equation}
and by  $\ppt=(f\vartheta_\epsilon)^{-1}\n$,
\begin{align*}
\he(\nabla^\epsilon f,\n)&=\he (\nabla^\epsilon f,(f\vartheta_\epsilon)^{-1}\ppt)
=(f\vartheta)^{-1}\frac{\partial f}{\partial t}.
\end{align*} 
Hence we conclude that \eqref{eq:general divergence} becomes
\begin{equation}\label{eq:simplified divergence}
	  \dve_S (\varphi \pps)= \Big [\varphi'(s)-\varphi fA^\epsilon(e_n^*,e_n^*)+\varphi (f\vartheta)^{-1}\frac{\partial f}{\partial t}\Big ]\cdot  |\he(e_n,\n)|^2+\varphi H^\epsilon f. 
\end{equation}

\medskip
If we define the vector field ($\beta$ is to be specified later)
\[Y^\epsilon:=(1-\beta(s))\exp(-Cs)\pps,\]
then from \eqref{eq:general divergence}, we have
\begin{align} 
&\ \ \ \ \dve_S Y^\epsilon\label{eq:divergence of Y epsilon}\\
&\leq\Big(\pps\Big[(1-\beta(s))\exp(-Cs)\Big] +(1-\beta(s))\exp(-Cs)\big[\he(\nabla^\epsilon f,\n)-f A^\epsilon(e_n^*,e_n^*)\big]\Big)\cdot|\he(e_n,\n)|^2+\nonumber\\
&+(1-\beta(s))\exp(-Cs)\cdot|H^\epsilon|f\nonumber\\ 
&\leq-\beta'(s)\cdot\exp(-Cs)|h_\epsilon(e_n,\n)|^2+\big(|H^\epsilon|f+|\he(\nabla^\epsilon f,\n)| \big).\nonumber
\end{align}
For the second inequality, we used that 
\[ \Big|(f\vartheta)^{-1}\frac{\partial f}{\partial t}-f A^\epsilon(e_n^*,e_n^*)\Big| \leq C.\]
 Since the varifold $V_\epsilon$ is $\he$-stationary with free boundary, for all $\epsilon>0$ small:
\[\delta V_\epsilon(Y^\epsilon)=\int \dve Y^\epsilon d V_\epsilon=0.\]

\medskip
Now we consider $\beta(s):\mb R\rightarrow [0,1]$ to be a non-decreasing function 
such that 
\begin{itemize}
\item $\beta(s)\equiv 0$ (resp. 1) when $s\leq -\wti R$ (resp. $s\geq 2\epsilon$);
\item on $[-\wti R,\epsilon]$, $\frac{\partial \beta}{\partial s}\geq 1/(2\wti R)$.
\end{itemize}
Here $\wti R$ is large enough so that $\spt V_\epsilon$ does not intersect $\{s<-\wti R\}$; see \eqref{eq:d is equiv to s}.

By the computation in \eqref{eq:divergence of Y epsilon}, for any $b>0$, we obtain the main result in this part:
\begin{align*}
&\ \ \ \ \int_{\ms F(T\times[0,2\epsilon])\times \mf G(n+1,n)}\chi_{\{|\he(e_n,\n)|>b\}}dV_\epsilon(x,S)\\
&\leq 2\wti R\exp(C\wti R)b^{-2}\int_{F(T\times[0,3\epsilon])\times \mf G(n+1,n)}|H^\epsilon|\cdot f\ dV_\epsilon(x,S)\nonumber\\
&\rightarrow 0,\  \text{ as }\ \ \epsilon\rightarrow 0.\nonumber
\end{align*}

\subsection{Proof of \eqref{eq:normal V'} }\label{subsec:proof of normal V'}
\begin{align*}
&\ \  \  \lim_{k\rightarrow\infty}\Big|\int\dv_S^{\epsilon_k}X_\perp^{\epsilon_k}dV_k'(x,S)-\int\dv^0X_\perp^0dV'_\infty\Big|\\
&=\lim_{b\rightarrow 0}\lim_{k\rightarrow\infty}\Big|\int \chi_{\{|h_{\epsilon_k}(e_n,\n)|\leq b\}}\dv_S^{\epsilon_k}X_\perp^{\epsilon_k}dV_k'(x,S)-\int\dv^0X^0_\perp dV'_\infty\Big|\\
&\leq \lim_{b\rightarrow 0}\lim_{k\rightarrow \infty}\Big|\int \chi_{\{|h_{\epsilon_k}(e_n,\n)|\leq b\}}\dv_{S_\perp}^0X_\perp^{\epsilon_k}dV_k'(x,S)-\int\dv^0X^0_\perp dV'_\infty\Big|+\\
&+\lim_{b\rightarrow 0}\lim_{k\rightarrow \infty}\int \chi_{\{|h_{\epsilon_k}(e_n,\n)|\leq b\}} 2|\nabla^{\epsilon_k}X^{\epsilon_k}_\perp|_{h_{\epsilon_k}}\cdot |e_n-e_n^*|dV_k'(x,S)\\
&=\lim_{k\rightarrow \infty}\Big|\int \dv^0_{S_\perp}X_\perp^{\epsilon_k}dV_k'(x,S)-\int\dv_{S_\perp}^0X^0_\perp dV'_\infty\Big|=0.
\end{align*}
Here the inequality is from \eqref{eq:bound Upsilon}.

\subsection{$|\nabla^\epsilon Z^\epsilon|_\he$ is uniformly bounded}\label{subsec:new vector has bounded gradient}
Recall that 
\[Z^\epsilon:=\varphi\nabla^\epsilon s=\varphi f^{-1}\n.\]
Then for $1\leq i,j\leq n-1$, 
\begin{gather*}
|\he(\nabla^\epsilon_{e_i}Z^\epsilon,e_j)|\leq |\varphi f^{-1}|\cdot |A ^{\epsilon}(e_i,e_j)|\leq |X^0_\parallel|_g,\\
|\he(\nabla^\epsilon_{e_j}Z^\epsilon,\n)|\leq |(\nabla^\epsilon(\varphi  f^{-1}))_\perp|_\he=|(\nabla^g(\varphi f^{-1}))_\perp|_g,\\
\he(\nabla^\epsilon_{\n}Z^\epsilon,\n)= \he(\nabla^\epsilon(\varphi f^{-1}),\n)=\he(\nabla ^\epsilon(\varphi f^{-1}),\vartheta^{-1}f^{-1} \ppt)=\vartheta^{-1}f^{-1}\ppt(\varphi f^{-1}),\\
|\he(\nabla^\epsilon_{\n}Z^\epsilon,e_j)|=|\he(f^{-1}\nabla^\epsilon_{\pps}Z^\epsilon,e_j)|=\big|\he(f^{-1}Z^\epsilon,\nabla^\epsilon_{e_j}(\pps))\big|\leq  |\varphi f^{-2}(\nabla f)_\perp|_g.
\end{gather*}

\subsection{Proof of \eqref{eq:parallel V''} }\label{subsec:parallel V''}
Let $H^\epsilon$ be the mean curvature as above. Recall that 
\[Z^\epsilon:=\varphi\nabla^\epsilon  s.\]
Then the divergence is
\begin{align}
\dve_SZ^\epsilon&=\dve_{S_\perp}Z^\epsilon+\he(\nabla^\epsilon_{e_n}Z^\epsilon,e_n)-\he(\nabla^\epsilon_{e_n^*}Z^\epsilon,e_n^*)\\
&=\he(Z^\epsilon,\n)\cdot H^\epsilon+\Upsilon'(\epsilon,x,S,X),\nonumber
\end{align}
where 
\begin{align*}
\big|\Upsilon'(\epsilon,x,S,X)\big|&=\big|\he(\nabla^\epsilon_{e_n}Z^\epsilon,e_n)-\he(\nabla^\epsilon_{e_n^*}Z^\epsilon,e_n^*)\big|\\
&\leq 2|\nabla^\epsilon Z^\epsilon|_{\he}\cdot|e_n-e_n^*|_{\he}.
\end{align*}
Recall that $h_{\epsilon_k}=g$ on $B_k$. Then we have 
\begin{align*} 
&\ \ \ \ \lim_{k\rightarrow\infty}\Big|\int\dv_S^{\epsilon_k}X_\parallel^{\epsilon_k}dV_k''(x,S)\Big|=\lim_{k\rightarrow\infty}\Big|\int \dv_S^{\epsilon_k}Z^{\epsilon_k}dV_k''(x,S)\Big|=\lim_{k\rightarrow\infty}\Big|\int \dv_S^{\epsilon_k}Z^{\epsilon_k}dV_k'(x,S)\Big|\\
&\leq \lim_{b\rightarrow 0}\lim_{k\rightarrow \infty}\int \chi_{\{|h_{\epsilon_k}(e_n,\n)|\leq b\}}|h_{\epsilon_k}(Z^{\epsilon_k},\n)\cdot H^\epsilon|+ 2|\nabla^{\epsilon_k}Z^{\epsilon_k}|_{h_{\epsilon_k}}\cdot |e_n-e_n^*|_{h_{\epsilon_k}}dV_k'(x,S)\\
&=0.
\end{align*}
Here the first equality comes from the fact that $X^{\epsilon_k}_\parallel=Z^{\epsilon_k}$ as in $B_k$; the second equality follows from that $V_k$ is stationary with free boundary; the last equality comes from Lemma \ref{lem:2 form and mean curvatue in new metric}. 
\bibliographystyle{amsalpha}
\bibliography{minmax}
\end{document}

%% file: both_contract.eps_tex
\begingroup%
  \makeatletter%
  \providecommand\color[2][]{%
    \errmessage{(Inkscape) Color is used for the text in Inkscape, but the package 'color.sty' is not loaded}%
    \renewcommand\color[2][]{}%
  }%
  \providecommand\transparent[1]{%
    \errmessage{(Inkscape) Transparency is used (non-zero) for the text in Inkscape, but the package 'transparent.sty' is not loaded}%
    \renewcommand\transparent[1]{}%
  }%
  \providecommand\rotatebox[2]{#2}%
  \ifx\svgwidth\undefined%
    \setlength{\unitlength}{314.1bp}%
    \ifx\svgscale\undefined%
      \relax%
    \else%
      \setlength{\unitlength}{\unitlength * \real{\svgscale}}%
    \fi%
  \else%
    \setlength{\unitlength}{\svgwidth}%
  \fi%
  \global\let\svgwidth\undefined%
  \global\let\svgscale\undefined%
  \makeatother%
  \begin{picture}(1,0.43944016)%
    \put(0,0){\includegraphics[width=\unitlength]{both_contract.eps}}%
    \put(0.31028664,0.25941274){\color[rgb]{0,0,0}\makebox(0,0)[lb]{\smash{$\Gamma_1$}}}%
    \put(0.66813819,0.21970247){\color[rgb]{0,0,0}\makebox(0,0)[lb]{\smash{$\Gamma_2$}}}%
    \put(0.96044906,0.12374207){\color[rgb]{0,0,0}\makebox(0,0)[lb]{\smash{$T$}}}%
    \put(0.51667906,0.00901354){\color[rgb]{0,0,0}\makebox(0,0)[lb]{\smash{$T$}}}%
    \put(0.0914658,0.22324029){\color[rgb]{0,0,0}\makebox(0,0)[lb]{\smash{$\wti F^1(\Gamma_1\times\{-\epsilon\})$}}}%
    \put(0.77346564,0.19518223){\color[rgb]{0,0,0}\makebox(0,0)[lb]{\smash{$\wti F^2(\Gamma_2\times\{-\epsilon\})$}}}%
    \put(0.50311748,0.22350052){\color[rgb]{0,0,0}\makebox(0,0)[lb]{\smash{$N'$}}}%
  \end{picture}%
\endgroup%

%% file: half_proper.eps_tex
\begingroup%
  \makeatletter%
  \providecommand\color[2][]{%
    \errmessage{(Inkscape) Color is used for the text in Inkscape, but the package 'color.sty' is not loaded}%
    \renewcommand\color[2][]{}%
  }%
  \providecommand\transparent[1]{%
    \errmessage{(Inkscape) Transparency is used (non-zero) for the text in Inkscape, but the package 'transparent.sty' is not loaded}%
    \renewcommand\transparent[1]{}%
  }%
  \providecommand\rotatebox[2]{#2}%
  \ifx\svgwidth\undefined%
    \setlength{\unitlength}{413.04663428bp}%
    \ifx\svgscale\undefined%
      \relax%
    \else%
      \setlength{\unitlength}{\unitlength * \real{\svgscale}}%
    \fi%
  \else%
    \setlength{\unitlength}{\svgwidth}%
  \fi%
  \global\let\svgwidth\undefined%
  \global\let\svgscale\undefined%
  \makeatother%
  \begin{picture}(1,0.43601707)%
    \put(0,0){\includegraphics[width=\unitlength]{half_proper.eps}}%
    \put(0.03876161,0.14812662){\color[rgb]{0,0,0}\makebox(0,0)[lb]{\smash{$T_0^2$}}}%
    \put(0.94853177,0.02427539){\color[rgb]{0,0,0}\makebox(0,0)[lb]{\smash{$T_0^2$}}}%
    \put(0.45988613,0.1626945){\color[rgb]{0,0,0}\makebox(0,0)[lb]{\smash{$\Gamma'''$}}}%
  \end{picture}%
\endgroup%